\documentclass[preprint,12pt]{elsarticle}

\makeatletter
\def\ps@pprintTitle{%
	\let\@oddhead\@empty
	\let\@evenhead\@empty
	\let\@oddfoot\@empty
	\let\@evenfoot\@empty
}

\usepackage[top=2cm,bottom=2cm,left=1.5cm,right=1.5cm,headheight=14pt,headsep=0.25in,footskip=0.5in]{geometry}

\usepackage{amsmath,amssymb,amsthm}
\usepackage[colorlinks,linkcolor=blue]{hyperref}
\usepackage[nameinlink, capitalize]{cleveref}
\usepackage{amssymb}
\usepackage{tabularx}
\usepackage{amsmath}
\usepackage{comment}
\usepackage{enumitem}
\newtheorem{thm}{Theorem}[section]
\newtheorem{prop}[thm]{Proposition}
\newtheorem{lem}[thm]{Lemma}
\newtheorem{cor}[thm]{Corollary}
\newtheorem{defin}[thm]{Definition}
\newtheorem{rem}[thm]{Remark}
\newtheorem{con}[thm]{Construction}

\newtheorem{nota}[thm]{Notation}
\newcommand{\hs}{\hookrightarrow}
\DeclareMathOperator{\PG}{PG}
\DeclareMathOperator{\AG}{AG}
\DeclareMathOperator{\rank}{rank}
\DeclareMathOperator{\CAN}{CAN}

\DeclareMathOperator{\CA}{CA}
\DeclareMathOperator{\OA}{OA}
\DeclareMathOperator{\Tr}{Tr}
\DeclareMathOperator{\ACM}{ACM}
\DeclareMathOperator{\ACTM}{ACTM}
\DeclareMathOperator{\CPHF}{CPHF}
\DeclareMathOperator{\SCPHF}{SCPHF}

\journal{Journal of Combinatorial Theory, Series A}
\begin{document}

\begin{frontmatter}

\title{Existence of $3$ anti-cocircular truncated M{\"o}bius planes and constructions of strength-4 covering arrays}

\author[1]{Kianoosh Shokri\corref{cor1}} %% Author name
\cortext[cor1]{Corresponding author}
\ead{kshokri@uottawa.ca}
\ead{contact@kianoosh-shokri.com}
\affiliation[1]{organization={Department of Mathematics and Statistics, University of Ottawa},Department and Organization
	addressline={150 Louis-Pasteur Private}, 
	city={Ottawa},
	postcode={K1N 9A7}, 
	state={Ontario},
	country={Canada}
   }

\author[2]{Lucia Moura} %% Author name

\affiliation[2]{organization={School of Electrical Engineering and Computer Science, University of Ottawa}, Department and Organization
	addressline={800 King Edward Avenue}, 
	city={Ottawa},
	postcode={K1N 6N5}, 
	state={Ontario},
	country={Canada}}

\author[3]{Brett Stevens} %% Author name

\affiliation[3]{organization={School of Mathematics and Statistics, Carleton University},Department and Organization
	addressline={1125 Colonel By Drive}, 
	city={Ottawa},
	postcode={K1S 5B6}, 
	state={Ontario},
	country={Canada}}

%% Abstract

\begin{abstract}

Two projective (affine) planes with the same point sets are \emph{orthogoval} if the common intersection of any two lines, one from each, has size at most two. The existence of a pair of orthogoval projective planes has been proven and published independently many times. 
A strength-$t$ \emph{covering array}, denoted by $\CA(N; t, k, v)$, is an $N \times k$ array over a $v$-set such that in any $t$-set of columns, each $t$-tuple occurs at least once in a row. 
A pair of orthogoval projective planes can be used to construct a strength-$3$ covering array $\CA(2q^3-1; 3, q^2 + q + 1, q)$. Our work extends this result to construct arrays of strength $4$. A \emph{$k$-cap} in a projective geometry is a set of $k$ points no three of which are collinear.
In $\PG(3,q)$, an \emph{ovoid} is a maximum-sized $k$-cap with $k =q^2+1$. Its plane sections (circles) are the blocks of a $3$-$(q^2 + 1, q + 1, 1)$ design, called a \emph{M{\"o}bius plane} of order $q$. For $q$ an odd prime power, we prove the existence of three truncated M{\"o}bius planes, such that for any choice of these circles, one from each plane, their intersection size is at most three. From this, we construct  
a strength-$4$ covering array $\CA(3q^4-2; 4, \frac{q^2+1}{2}, q)$. 
For $q \geq 11$, these covering arrays improve the size of the best-known covering arrays with the same parameters by almost 25 percent. The $\CA(3q^4 -2; 4, \frac{q^2 +1}{2}, q)$ is used as the main ingredient in a recursive construction to obtain a $\CA(5q^4 - 4q^3 - q^2 + 2q; 4, q^2 +1, q)$. Some improvements are obtained in the size of the best-known arrays using these covering arrays.

\end{abstract}

%% Keywords

\end{frontmatter}
\section{Introduction}

In a projective plane $\PG(2,q)$, an \emph{oval} is a set of $q+1$ points, no three of which are collinear. Two projective (affine) planes with the same point sets are \emph{orthogoval} if a line in one is an oval in another \cite{Colbourn2022Brett}. Equivalently, two projective (affine) planes with the same point sets are orthogoval if the common intersection of any line from each has size at most two. A set of planes is called a \emph{set of mutually orthogoval planes} if the planes are pairwise orthogoval. The existence of orthogoval projective planes is proved independently in \cite{Bruck1973, Glynn1978,Hall1947,Dieter1984,Raaphorst2014,Veblen}, and non-independent novel proofs are published in \cite{Baker1994,Colbourn2022Brett,Panario2020}. 
Colbourn et al. \cite{Colbourn2022Brett} provide a new proof for the existence of a pair of orthogoval projective planes and show the existence of orthogoval affine planes for even $q$ and $q=3$. Only when $q =3$ can a pair of orthogoval affine planes be extended to a pair of orthogoval projective planes. However, it is possible to remove $2(q-1)$ points from a pair of orthogoval projective planes for any prime power $q$ to obtain truncated orthogoval affine planes \cite{torres2018}. The methods used to construct a pair of orthogoval projective planes in the literature are diverse. For example, Raaphorst et al. \cite{Raaphorst2014} use a primitive polynomial and its reciprocal to construct a pair of orthogoval projective planes, while Baker et al. \cite{Baker1994} use primitive elements in $\mathbb{F}_{q^3}$; both use $\alpha$ to generate the points of the first projective plane, while the second projective plane uses $\alpha^{-1}$ for any prime power $q$ \cite{Baker1994, Raaphorst2014}, and $\alpha^{2}$ or $\alpha^{\frac{1}{2}}$ for an odd prime power $q$ \cite{Baker1994}. An extension of the $\alpha^{-1}$ method is studied in \cite{Panario2020}, where $\alpha$ is a root of an arbitrary polynomial, not necessarily primitive.  

A pair of orthogoval projective (affine) planes can be used to construct a strength-$3$ covering array (see \cref{subsec:ca} for the definition) $\CA(2q^3 -1; 3, q^2 +q +1,q)$ \cite{Raaphorst2014}. The possibility of generalizing or extending the definition of orthogoval planes to other geometric objects in higher dimensions, such as \emph{M{\"o}bius planes}, is mentioned in  \cite{Colbourn2022Brett}. A strength-$4$ covering array $\CA(511; 4, 17, 4)$ obtained in \cite{Tzanakis2016} suggested a connection with \emph{ovoids} in $\PG(3,q)$ for $q=4$. An ovoid in $\PG(3,q)$ is a set of $q^2 +1$ points, no three of which are collinear. 
A \emph{$t$-$(v, k, \lambda)$ design} is a pair $(X, \mathcal{B})$ where $X$ is a $v$-set of points and $\mathcal{B}$ is a collection of $k$-subsets of $X$ (blocks) with the property that every $t$-subset of $X$ is contained in exactly $\lambda$ blocks. The parameter $\lambda$ is the \emph{index} of the design. The plane sections of the ovoid are the blocks of a $3$-$(q^2 +1 , q+1, 1)$ design, also called a M{\"o}bius plane \cite{okeefe1996}. 
Let $(\mathcal{X}, \mathcal{B})$ be a $t$-$(v, k, \lambda)$ design. A \emph{truncated} design is obtained from $(\mathcal{X}, \mathcal{B})$ by removing some of its points, $\mathcal{P} \subseteq \mathcal{X}$, and getting $(\mathcal{X'} = \mathcal{X} \setminus \mathcal{P} , \mathcal{B'} = \{ B \cap \mathcal{X'}: B \in \mathcal{B}\} )$. 
We provide an extension of the definition of orthogoval planes. A set of $m$ (truncated) M{\"o}bius planes on the same point set is \emph{anti-cocircular} if the common intersection of any choice of $m$ blocks one from each of the planes is at most size three. In other words, any four points cannot be simultaneously contained in $m$ circles, one in each of the M{\"o}bius planes. A set of $m$ anti-cocircular M{\"o}bius planes of order $q$ is denoted by $\ACM(m,q)$. If those M{\"o}bius planes are truncated, we denote it by $\ACTM(m,q)$. 
In this paper, we have two objectives. 
First, we show the existence of an $\ACTM(3,q)$; and second, we build a $\CA(3q^4 -2; 4, \frac{q^2 +1}{2}, q)$ for any odd prime power $q$ by vertically concatenating three strength-$3$ covering arrays. Each strength-$3$ covering array is obtained by a generator matrix whose columns correspond to one of the three truncated M{\"o}bius planes. For $q \geq 11$, this covering array improves the size of the best-known covering arrays with the same parameters by almost 25 percent (see \cref{table:4ca3layer}). \cref{tab:sum} shows the geometric objects and corresponding constructed covering arrays. 

\begin{table}[h]
    \centering
    \scalebox{0.72}{
    \begin{tabular}{||c|c|c|c|c||}
    \hline \hline
        $q$ & Geometric object & Corresponding CA & References\\
         \hline \hline 
        All prime powers & Pair of orthogoval $\PG(2,q)$ &$\CA(2q^3 -1; 3, q^2 + q +1, q)$ & \cite{Bruck1973, Glynn1978,Hall1947,Dieter1984,Raaphorst2014,Veblen}\\
         \hline 
        Even prime powers & Pair of orthogoval $\AG(2,q)$ & $\CA(2q^3 -q; 3, q^2 + 2, q)$ & \cite{Colbourn2022Brett}\\
         \hline 
        All prime powers & Pair of truncated orthogoval $\AG(2,q)$ & $\CA(2q^3 -q; 3, q^2 - q +3, q)$& \cite{torres2018}\\
         \hline 
         $q = 3$ & Set of 4 mutually orthogoval $\PG(2,3)$ & $\CA_{\lambda}(27\lambda + 26 ; 3, 13 , 3)$, $\lambda < 4$& \cite{Colbourn2022Brett}\\
          \hline 
          $q = 3$ & Set of 7 mutually orthogoval $\AG(2,3)$ & $\CA_{\lambda}(27\lambda + 24 ; 3, 11 , 3)$, $\lambda < 7$& \cite{Colbourn2022Brett}\\
          \hline 
          $q = 2^n$ , $\gcd(n, 6) = 1$ & Triple of mutually orthogoval $\AG(2,q)$ & $\CA_{\lambda}(\lambda q^3 + q^3 - q ; 3, q^2 +2 , q)$, $\lambda < 3$& \cite{Colbourn2022Brett}\\
          \hline 
          $q = 4$ & Set of $7$ mutually orthogoval $\AG(2,4)$ & $\CA_{\lambda}(64\lambda + 60 ; 3, 18 , 4)$, $\lambda < 7$& \cite{Colbourn2022Brett}\\
          \hline 
          $q = 8$& Set of $7$ mutually orthogoval $\AG(2,8)$ & $\CA_{\lambda}(512\lambda + 504; 3, 66 , 8)$, $\lambda < 7$& \cite{Colbourn2022Brett}\\
         \hline 
         
        $q = 4$ & $\ACM(2,4)$  & $\CA(511; 4, 17, 4)$ & \cite{Tzanakis2016}\\
         \hline 
        Odd prime powers & $\ACTM(3,q)$ & $\CA(3q^3 -2; 4, \frac{q^2 +1}{2}, q)$ & \cref{orthogovoid}, \cref{4ca3layer}\\
         \hline \hline 
    \end{tabular}}
    \caption{Different geometric objects and corresponding constructed covering arrays}
    \label{tab:sum}
\end{table}
In the construction of the $\CA(3q^4 -2; 4, \frac{q^2 +1}{2}, q)$, we use $\frac{q^2 +1}{2}$ points of the ovoid (half of the points of the ovoid). We observed that half ovoids have been used in other contexts, such as the construction of a family of projective two-weight codes (strongly regular Cayley graphs) in an unpublished work \cite{wilson1997}.

Extending our construction to employ all $q^2 +1$ points of the ovoid is feasible using additional geometry and recursive constructions.
Using these, we construct a $\CA(5q^4 - 4q^3 - q^2 + 2q; 4, q^2 +1, q)$ for any odd prime power $q$. These covering arrays improve the size of many covering arrays with the same parameters. 

The structure of this paper is as follows. In \cref{sec:bg}, we review the background in finite geometry, covering arrays, and their connections. In \cref{sec:orthogoval}, we review the existence of pairs of orthogoval projective planes in the literature and their role in constructing strength-$3$ covering arrays.
In \cref{sec:orthogovoid}, we give results connecting blocks of truncated M{\"o}bius planes with hypersurfaces in $\PG(3,q)$. We prove the existence of an $\ACTM(3,q)$ for any odd prime power $q$. 
In \cref{sec:4ca3layer}, we prove the existence of a $\CA(3q^4 -2; 4, \frac{q^2 + 1}{2}, q)$ obtained from the vertical concatenation of three strength-$3$ covering arrays obtained by an $\ACTM(3,q)$ for any odd prime power $q$. In \cref{sec:rec4-ca}, we prove the existence of a $\CA(5q^4 - 4q^3 - q^2 + 2q; 4, q^2 +1, q)$ using a recursive method with the $\CA(3q^4 -2; 4, \frac{q^2 +1}{2}, q)$ as the main ingredient. In \cref{sec:future}, we discuss potential future works. 

\section{Necessary Background on Finite Geometry and Covering Arrays} \label{sec:bg}
In this section, we provide the necessary background on finite geometry and its connections with combinatorial designs, especially with covering arrays. 
\subsection{Background on geometry}
We start this section with the basics of finite fields. 
Let $q=p^m$, where $p$ is a prime and $m \geq 1$, and let $\mathbb{F}_{q}$ denote the \emph{finite field} of $q$ elements.
An element $w \in \mathbb{F}_{q}$ is \emph{primitive} if $w$ generates the multiplicative group of $\mathbb{F}_{q}$: $\mathbb{F}_{q} = \{ 0, 1, w, w^{2}, \ldots , w^{q -2}\}$.
A polynomial $f \in \mathbb{F}_p[x]$ of degree $m$ is \emph{primitive} if $f$ is monic, and has a root $\alpha \in \mathbb{F}_{p^m}$ that is primitive.

\begin{defin} \label{defin:trace}
	The \emph{trace function}  is defined as follows: 
    \begin{equation}
	\begin{aligned}
		\Tr : & \ \mathbb{F}_{q^m} \rightarrow \mathbb{F}_{q},\\ 
		 & \Tr(a) =  a + a^q + a^{q^2} + \cdots + a^{q^{m-1}}. 
	\end{aligned}
	\end{equation}
	The trace function is $\mathbb{F}_q$-linear which means that for all $s, t \in \mathbb{F}_q$ and all $a, b \in \mathbb{F}_{q^m}$, $\Tr(sa+tb) = s\Tr(a) + t\Tr(b)$.
\end{defin}

\begin{thm} [{\cite[Theorem 4.11]{golomb_gong_2005}}] \label{p1lfsr}
	Let $f(x) = x^m + \sum_{j = 1}^{m} b_j x^{m-j}$ be a degree-$m$ primitive polynomial over $\mathbb{F}_q$ with root $\alpha$. Then, for any tuple $T = (a_0, \ldots , a_{m-1})$ of initial values, there exists a unique element $\beta \in \mathbb{F}_{q^m}$ such that $a_i = \Tr(\beta \alpha^i)$ for all $ 0 \leq i \leq m-1$. Furthermore, since the trace is $\mathbb{F}_q$-linear, it follows that the recurrence relation $a_i = \gamma_n = - \sum_{j  =1}^{n} b_j \gamma_{n-j}$ with an arbitrary non-zero tuple $T$ of initial values has the property that for all $i \geq 0$, $a_i = \Tr(\beta \alpha ^i)$.
\end{thm}

\begin{rem} \label{subprim}
    Let $\alpha$ be a primitive element in $\mathbb{F}_{q^m}$. Since $ q -1 \ | \ q^m -1$, then $\alpha^{\frac{q^m -1}{q-1}}$ is a primitive element in $\mathbb{F}_q$. 
\end{rem}

  A \emph{finite projective space} (or finite projective geometry) $P$ is a finite set of \emph{points} together with subsets of points called \emph{lines} satisfying the following properties:
        (1) any two distinct points are on exactly one line; 
        (2) let $A$, $B$, $C$, $D$ be four distinct points of which no three are collinear. If the lines $AB$ and $CD$ intersect each other, then the lines $AD$ and $BC$ also intersect each other; 
        (3) any line has at least three points. 
Let $V$ be a vector space of dimension $m$ over the finite field $\mathbb{F}_q$ for a prime power $q$. The geometry $\PG(m-1, q)$ that has its points the $1$-dimensional subspaces of $V$ and as its lines the $2$-dimensional subspaces of $V$ is a finite projective space of dimension $m-1$. There are $\frac{q^{m} -1}{q-1}$ points and $\frac{q^{m}-1}{q-1}$ hyperplanes with $\frac{q^{m-2}-1}{q-1}$ points each in $\PG(m-1, q)$. 
We represent a point in $\PG(m-1,q)$ in \emph{homogeneous coordinates} $(a_0 : a_1 : \dots : a_{m-1}) = \left \{ (ba_0 , ba_1 , \ldots , ba_{m-1}) : b \in \mathbb{F}_{q} \right \}$.
Let $P$ be a projective space of dimension at least two. A \emph{collineation} of $P$ is a bijective map $\alpha$ on the point set of $P$ that preserves collinearity; that is, $p$, $q$, $r$ is collinear if and only if $\alpha(p)$, $\alpha(q)$, $\alpha(r)$ is collinear. A \emph{Singer cycle} of a finite projective space $\PG(m - 1, q)$ is a collineation $\delta$ such that the cyclic group generated by $\delta$ acts regularly on the points (and hyperplanes) of $\PG(m-1,q)$. This group is called a \emph{Singer group} \cite{Keldon2002}.
A \emph{$k$-cap} of $\PG(m-1, q)$ is a set of
$k$ points of $\PG(m-1, q)$, no three of which are collinear. A $k$-cap in $\PG(2, q)$ is called a \emph{$k$-arc}. In $\PG(3,q)$, an \emph{ovoid} is a $k$-cap with maximum size. 
Since no three points of a $k$-cap $\mathcal{K}$ of $\PG(m-1,q)$ are collinear, the lines of $\PG(m-1,q)$
fall into three classes with respect to $\mathcal{K}$. An \emph{external line} contains no point of $\mathcal{K}$,
a \emph{tangent line} contains exactly one point of $\mathcal{K}$, and a \emph{secant line} contains exactly two
points of $\mathcal{K}$. 

Bose \cite{Bose1947} for odd $q > 2$, Seiden \cite{Seiden1950} for $q = 4$, and Qvist \cite{Qvist1952} for all even $q > 2$ showed that the maximum value of $k$ such that there exist a $k$-cap (ovoid) in $\PG(3, q)$ is $q^2 + 1$.
Barlotti \cite{barlotti1955} and Panella \cite{panella1955} show that for an ovoid $\mathcal{K}$ in $\PG(3,q)$ with $q > 2$: (1) at each point $P$ of $\mathcal{K}$, the $q + 1$ tangent lines to  $\mathcal{K}$ lie in a plane, called the tangent plane to $\mathcal{K}$ at $P$. Thus $\mathcal{K}$ admits $q^2 + 1$ \emph{tangent planes}; (2) every plane of $\PG(3, q)$ is either a tangent plane or else meets $\mathcal{K}$ in a $(q + 1)$-arc, in which case it is called a \emph{secant plane} of $\mathcal{K}$.  
The points and hyperplanes of $\PG(m-1, q)$ form a $2$-$(\frac{q^m-1}{q-1}, \frac{q^{m-1}-1}{q-1}, \frac{q^{m-2}-1}{q-1})$ design. 
 A \emph{M{\"o}bius plane} of order $q$ is a $3$-$(q^2 + 1, q + 1, 1)$ design. Given an ovoid $\mathcal{K}$ in $\PG(3, q)$, the incidence structure with \emph{points} being the points of the ovoid, \emph{blocks} or
\emph{circles} being the secant plane sections of the ovoid, and the incidence that of $\PG(3, q)$ is a
M{\"o}bius plane of order $q$. O'Keefe \cite{okeefe1996} provided a survey about ovoids in $\PG(3,q)$, where classifications and characterization of ovoids are studied. More recently, the uniqueness of the M{\"o}bius plane of order 64 has been established in \cite{Mpenttila2022}.

Throughout this paper, let $\alpha$ be a primitive element of $\mathbb{F}_{q^4}$, and set $\gamma = \alpha^\frac{(q^2+1)(q+1)}{2}$, for an odd prime power $q$. Note that $\gamma^q = -\gamma$, and $\gamma^{q^2} = \gamma$. As a consequence, $\gamma \in \mathbb{F}_{q^2}$, $\gamma^2 \in \mathbb{F}_q$, and $\Tr(\gamma) = 0$. Any point of $\PG(3,q)$ is denoted by $\langle x\rangle$, where
$\langle x\rangle=\{\lambda x:\lambda\in\mathbb F_q^{*}\}$
is the set of all nonzero scalar multiples of $x\in\mathbb F_{q^4}^{*}$. Let $\delta: \PG(3,q) \rightarrow \PG(3,q)$ be the map where $\delta(\langle \alpha^{\ell} \rangle) = \langle \alpha^{\ell + 1} \rangle$ for $0 \leq \ell < (q^2 +1)(q+1)$. Then $\delta$ is a Singer cycle. 
Let $\mathcal{O} = \{ \langle \alpha^{i(q +1)} \rangle$: $0 \leq i < q^2 +1 \}$. For all prime powers $q$, Ebert proved that the set $\mathcal{O}$ is an ovoid by showing that any three elements of $\mathcal{O}$ are linearly independent over $\mathbb{F}_q$ (see \cref{genovoid}). For an odd prime power $q$, which is the focus of this paper, the following approach can be used to show this result. 
 Let $B: \mathbb{F}_{q^4}^2 \rightarrow \mathbb{F}_{q^4}$ with $B(X,Y) = \Tr(\gamma XY^{q^2})$ be a $\mathbb{F}_q$-bilinear form with the corresponding quadratic form $f(X) = \Tr(\gamma X^{q^{2} + 1})$. Let $X \in \mathbb{F}_{q^4}$; then $B(X, Y) = 0$ for all $Y \in \mathbb{F}_{q^4}$ if and only if $X=0$. This is simply saying that $f$ is a non-degenerate quadric \cite{Ball2015}. The set $\mathcal{O} = \{ \langle \alpha^{i(q +1)} \rangle$: $0 \leq i < q^2 +1\}$ is the kernel of $f(X)$. 
 Since $f$ is a non-degenerate quadric of $\PG(3,q)$ with $q^2 + 1$ points, it is an elliptic quadric and $\mathcal{O}$ is an ovoid \cite{Cossidente2022}. So, $\Tr(\gamma X^{q^2 +1}) = 0$ represents an ovoid in $\PG(3,q)$.  The equation of any plane in $\PG(3, q)$ can be written as $\Tr(b X) = 0$ for some $b \in \mathbb{F}_{q^4}^{*}$. Any $\mathbb{F}_q$-multiple of $b$ yields the same plane; thus, there are $(q^2 + 1)(q + 1)$ planes, of which $q^2 + 1$ are tangent to $\mathcal{O}$, and the rest are secant.
 Since $B(X, Y)$ is the non-degenerate bilinear form associated with the quadratic form $f(X)$, the tangent plane to $\mathcal{O}$ at $\langle y\rangle$ has equation 
 $B(X,y) = \Tr(\gamma X y^{q^2}) = 0$ \cite{Bart2024}. In particular, the tangent plane at $\langle 1 \rangle$ and $\langle \gamma \rangle$ have equations $\Tr(\gamma X) = 0$ and $\Tr(X)= 0$, respectively. Note that for any $b \in \mathbb{F}_{q^4}$ with $\Tr(b) = 0$ , $\Tr(b X) = 0$  gives a plane containing $\langle 1 \rangle$; when $b \neq \langle \gamma \rangle$, then this plane is secant to $\mathcal{O}$. It is enough to study planes passing through $ \langle 1 \rangle$, since all other planes can be obtained from these under the action of the Singer cycle.

\begin{thm} [{\cite[Theorem 3]{Ebert1985}}] \label{genovoid}
    Let $q$ be a prime power. The set $\mathcal{O} = \lbrace \langle \alpha^{x(q+1)} \rangle : 0 \leq x < q^2 + 1 \rbrace$ is an ovoid in $\PG(3,q)$. 
\end{thm}

In \cref{Gengen}, a generator matrix is given whose columns correspond to $c$ distinct points of $\PG(m-1, q)$.
\begin{con}\label{Gengen}
Let $\mathbb{F}_{q^m}$ be a finite field for a prime power $q$ and $\alpha$ a primitive element of $\mathbb{F}_{q^m}$. For $\alpha^j \in \mathbb{F}_{q^n}$, suppose $\alpha^j = \sum_{n =0}^{m-1} a_n \alpha^n$. Let $L(\alpha^{j}) = \left (a_0 : a_1 :  \ldots : a_{m-1}\right)$ be the homogeneous coordinates representing $\langle \alpha^j \rangle$. 

	Let $G^{\ell}_{c}$ be the matrix where the $i$th column is $L(\alpha^{\ell i})$ for $0 \leq i < c$, where $ 0 < c \leq \frac{q^m -1}{q-1}$, and~$\ell$ is a positive integer,  
	
	\begin{equation}
		G^{\ell}_{c} := \begin{bmatrix}
			L(\alpha^0)^T & L(\alpha^{\ell})^T & L(\alpha^{2\ell})^T & \ldots & L(\alpha^{(c-1)\ell})^T \\
		\end{bmatrix}.
	\end{equation}
	
	Denote by $A(G^{\ell}_{c})$ the $q^m \times c$ matrix consisting of the row span of $G_c^{\ell}$ over $\mathbb{F}_q$.
\end{con}  

Properties of $G_c^{\ell}$ have been studied since \cite{Panario2020, Raaphorst2014} (see also \cite{Moura2016}). Let $\bar{c} = \frac{q^m -1}{q-1}$. As observed in  \cite{Panario2020, Raaphorst2014}, $\bar{c}$ is the maximum $c$ such that $G_c^{\ell}$ does not contain any two multiple columns, since the power $\frac{q^m - 1}{q-1}$ is the smallest non-zero power such that $\alpha^{\frac{q^m - 1}{q-1}} \in \mathbb{F}_q$. It is easy to see that $G_{\bar{c}}^1$ has rank $m$. In addition, since any two columns of $G^1_{\bar{c}}$ are linearly independent and there are $\frac{q^m -1}{q-1}$ columns, then the columns are in correspondence with the points of $\PG(m-1, q)$. Let $z = \frac{q^{m-1} - 1}{q -1}$ and $S = \lbrace i_1, i_2, \ldots i_z \rbrace$ be a set of column indices of $A = A(G_{\bar{c}}^1)$ with $A_{r,i_1} = \cdots = A_{r,i_z} = 0$ for some non-zero row $r$ of $A(G_{\bar{c}}^1)$. The points of $\PG(m-1, q)$ corresponding to $S$ consist of the $z$ points contained in a hyperplane of $\PG(m-1, q)$.

In the main results of this paper, we use $G^{q+1}_{q^2+1}$ in \cref{Gengen} for $m=4$. The columns of $G^{q+1}_{q^2+1}$ correspond to  points of an ovoid in $\PG(3, q)$ (\cref{genovoid}) and its plane sections corresponds to a M{\"o}bius plane. The rank of any $4 \times 3$ submatrix of $G^{q+1}_{q^2+1}$ is 3. The rank of any $4 \times 4$ submatrix of $G^{q+1}_{q^2+1}$ is either 3 or 4, depending on whether the corresponding $4$ points are coplanar or not, respectively. 

  \begin{rem} \label{3design}
    Let $q$ be a prime power and $m = 4$. Let $\mathcal{M} = \lbrace i: 0 \leq i < q^2 +1 \rbrace$ be the set of column indices of $G^{q+1}_{q^2+1}$. The columns of $G^{q+1}_{q^2+1}$ correspond to an ovoid $\mathcal{O}$ in $\PG(3,q)$. Its plane sections correspond to the blocks of a $3$-$(q^2+1, q+1, 1)$ design or a M{\"o}bius plane. 
    Each block (circle) $B$ of the $3$-$(q^2+1, q+1, 1)$ design (M{\"o}bius plane) corresponds to the column indices $\lbrace i_1, \ldots, i_{q+1} \rbrace$ with $A_{r,i_1} = \cdots = A_{r, i_{q+1}} = 0$ for some non-zero row $0 \leq r \leq q^4 -1$ of $A = A(G^{q+1}_{q^2+1})$.
    Therefore, the set: 
    \begin{center}
        $\mathcal{C} = \lbrace \lbrace i_1, \ldots, i_{q+1 } \rbrace : A_{r, i_1}= \cdots =  A_{r, i_{q+1}} = 0, \ 0 \leq r \leq q^4 - 1 \rbrace$
    \end{center}
    is the set of blocks of a $3$-$(q^2+1, q+1, 1)$ design or M{\"o}bius plane $(\mathcal{M}, \mathcal{C})$. 
\end{rem}

In the rest of this section, we review geometric objects and some of their properties in projective geometry. The reviewed materials can be found in \cite{Casse2006}. 

    Let $n$ be a positive integer. A \emph{homogeneous} polynomial $\phi$ of degree $n$ in the variables $x_0, x_1, \ldots, x_r$ over a field $F$ is the sum of terms of type $ax_0^{n_0}x_1^{n_1} \ldots x_r^{n_r}$ where $a \in F$, each $n_i$ is a non-negative integer, and $n_0 + \cdots + n_r = n$. Each term $ax_0^{n_0}x_1^{n_1}\ldots x_r^{n_r}$ of a homogeneous polynomial of degree $n$ is said to be of \emph{degree} $n$, and $a$ is called its \emph{coefficient}. 
    A \emph{hypersurface} $H$ of \emph{order} $n$, in $\PG(r,F)$, is a set of points $(x_0, x_1, \ldots, x_{r})$ satisfying $\phi(x_0, x_1, \ldots, x_{r}) = 0$, where $\phi$ is a non-zero homogeneous polynomial of degree $n$. 
     When $n=2$, the hypersurface is called a \emph{quadric}, and when $n=4$, the hypersurface is called a \emph{quartic}. 
     In $\PG(2,F)$, a hypersurface is referred to as a \emph{curve}, and when $n=2$, the plane curve is called a \emph{conic}. 
     If $\phi$ is irreducible over any extension of $F$, then the hypersurface $\phi = 0$ is said to be \emph{irreducible}; otherwise the hypersurface is \emph{reducible}.

\begin{nota}
Let $S$ be a set of points in $\PG(r, F)$ that satisfy the equation $\phi = 0$. We denote $S \hs \phi = 0$. 
\end{nota}

    Let $Q \hs \phi = 0$ be a quadric in $\PG(r,F)$ with $\phi = \sum_{j \geq i}a_{ij}x_ix_j$, $(0 \leq i \leq j \leq r)$ in $x_0, x_1, \ldots, x_r$. The quadric $Q$ is \emph{singular} if it contains at least one point $p$ such that every line through $p$ intersects $Q$ doubly at $p$; in other words, the equation of the quadric corresponding to the intersection of $Q$ and the line has a root with multiplicity equal to $2$. Such a point $p$ is called a \emph{singular point}.
    A point $p$ in a quadric $Q$ in $\PG(r,F)$ with equation $\phi = 0$ is singular if and only if $p$ satisfies $\frac{\partial \phi}{\partial x_i} = 0$, for $i = 0, 1, \ldots, r$. If $\mathrm{char}(F) \neq 2$, then equations $\frac{\partial \phi}{\partial x_i} = 0$, $i = 0, 1, \ldots, r$ can be written in matrix form as $Mx = 0$, where $M = [m_{ij}]$; $m_{ii} = 2a_{ii}$, $m_{ij} = m_{ji} = a_{ij}$.  
The $(r + 1) \times (r+1)$ matrix $M$ is the \emph{matrix associated with the quadric} defined by $\phi = 0$. Note that $\phi$ can be displayed using $M$: $\phi(x) = \frac{1}{2}x^T M x$ where $x^T = \left [x_0, x_1, \ldots, x_r \right ]$. It is clear that $\rank(M) = r +1$ if and only if $Mx = 0$ has a trivial solution. Then $\rank(M) = r + 1$ if and only if $Q$ is non-singular.

Some important properties of quadrics are listed in the following: 

\begin{enumerate}
    \item  A conic $C$ in $\PG(2,F)$ is reducible if and only if $C$ is singular. 

    \item  For a prime power $q$, if a non-singular conic has one point in $\PG(2,q)$, then it has precisely $q +1$ points in $\PG(2,q)$. 
     \item   Any five distinct points, no three collinear in $\PG(2,q)$ determine a unique conic. 
    \item  If $Q$ is a non-singular quadric in $\PG(3,q)$, then it is irreducible. We prove the contrapositive. Let $Q \hs \phi = 0$ be a reducible quadric. Then $\phi = \phi_1 \phi_2$ where $\phi_1$ and $\phi_2$ are equations of planes. If the two planes are distinct ($\phi_1 \neq \phi_2$), they intersect in a line, and all points on that line are singular. If two planes coincide ($\phi_1 = \phi_2$), then all points on the plane are singular. Therefore, $Q$ is singular. 

    \item  For $r \geq 3$, there are singular quadrics in $\PG(r,F)$ which are irreducible. 

    \item  Suppose $\mathrm{char}(F) \neq 2$. Let $K: \PG(r, F) \longrightarrow \PG(r, E)$ be a linear collineation of $\PG(r, F)$ where $E$ is a field extension of $F$ represented by matrix $K \in E^{(r+1) \times (r+1)}$. Suppose $\bar{x} = Kx$ denotes the new coordinates after the transformation $K$. Let $Q \hs \phi = 0$ be a quadric in $\PG(r, F)$ with associated matrix $M$, and $Q' \hs \phi' = 0$ be a quadric in $\PG(r, E)$ with associated matrix $M'$ after transforming $\phi$ using $K$. So, $\phi = \frac{1}{2} x^T M x$, and $\phi' = \frac{1}{2} \bar{x}^T M' \bar{x}$. It is clear that $M = K^T M' K$. Since $\det(M) = \det(K^T)\det(M')\det(K)$ and $\det(K) = \det(K^T) \neq 0$, then $\det(M) = 0$ if and only if $\det(M') = 0$. Therefore, $Q$ is singular if and only if $Q'$ is singular.  
\end{enumerate}
\begin{rem} \label{rm:omegat}
We are interested in studying polynomial hypersurfaces, the ovoid $\mathcal{O}$, and, its tangent and secant planes in $\PG(3,q)$ for $q$ odd. For this purpose, we introduce a coordinatization $\Omega (\langle X \rangle) =  [ X, X^q, X^{q^2}, X^{q^3} ]^T$, $\Omega: \mathbb{F}_{q^4} \longrightarrow \mathbb{F}_{q^4}^4$. 
Then, $\Omega$ is a linear transformation that diagonalizes the Singer cycle $\delta$ \cite{Baker1994}.

As we mentioned, $\mathcal{O} \hs \Tr(\gamma X^{q^2 + 1}) = 0$. Note that $\Tr(\gamma X^{q^2 + 1}) = 2\gamma (X^{q^2 + 1} - X^{q^3 + q}) = 0$ if and only if $X^{q^2 + 1} - X^{q^3 + q} = 0$. Let $ \mathcal{O}_{\omega} = \{ \Omega(\langle X \rangle) : X \in \mathcal{O} \}$. Then, in the new coordinatization, the ovoid $\mathcal{O}$ becomes $\mathcal{O}_{\omega} \hs xz - yt = 0$.   In addition, the equation of the plane corresponding to $\Tr(b X) = 0$ becomes $b x + b^{q} y + b^{q^2} z + b^{q^3} t = 0$. Since $ \mathcal{O}_{\omega}$ is non-singular, the intersection with a secant plane $P$ gives a non-singular conic. 
Substituting $t = \frac{xz}{y}$ into $b x + b^{q} y + b^{q^2} z + b^{q^3} t = 0$ yields 

    \begin{equation} \label{ovoidconic}
        \mathcal{O}_c \hs b^{q}y^2 + b^{q^2} yz  + b^{q^3} xz + b xy  = 0,
    \end{equation}
    which is a non-singular conic when $b \neq \langle \gamma \rangle$. When $b = \langle \gamma \rangle $, then $P$ is tangent at $\langle 1 \rangle$ and $\mathcal{O}_c = P \cap \mathcal{O}_{\omega}$ is a singular conic. 
\end{rem}

   Let $Q_1 \hs \phi_1 = 0$ and $Q_2 \hs \phi_2 =0$ be two quadrics in $\PG(r, F)$. Let $F' = F \cup \lbrace \infty \rbrace$. Then the set of quadrics $Q_1 + \lambda Q_2 \hs \phi_1 + \lambda \phi_2 = 0$, $\lambda \in F'$, where $\lambda = \infty$ corresponds to $Q_2 \hs \phi_2 =0$, is called a \emph{pencil of quadrics}.
    Any point of $Q_1 \cap Q_2$ satisfies both $\phi_1 = 0$ and $\phi_2 = 0$ and therefore lies on every quadric of the pencil $Q_1 + \lambda Q_2$. These points are the only such points that lie on every quadric of the pencil.
Let $C$ be a conic in $\PG(2,q)$ with equation $\phi(x,y, z) = ax^2 + by^2 + cz^2 + dyz + exz + fxy = 0$. Let $p_1$, $p_2$, $p_3$, and $p_4$ be four distinct points satisfying this equation, no three collinear. We can always define a linear transformation $T$ such that $T(p_1) = (1 : 0 : 0)$, $T(p_2) = (0 : 1 : 0)$, $T(p_3) = (0 : 0 : 1)$, and 
    $T(p_4) = (1 : 1 : 1)$. Then, the new equation of $C$ with respect to the new coordinates is $\phi' = d'y'z' + e'x'z' + f'x'y' = 0$, since $\phi'(1, 0, 0) = \phi'(0, 1, 0) = \phi'(0, 0, 1) = 0$ force the coefficients of $x'^2$, $y'^2$, and $z'^2$ to be equal to zero, and $d' + e' + f' = 0$.
\begin{prop} \label{onepencil}
    The only conics through four distinct points in $\PG(2,q)$ (no three collinear) are the members of a pencil of conics. 
\end{prop}

\begin{proof}
    Without loss of generality, let $C_1 \hs \phi_1 = 0 $, $C_2 \hs \phi_2 = 0$, and $C_3 \hs \phi_3 = 0$ be three conics sharing $(1 : 0 : 0)$, $(0 : 1 : 0)$, $(0 : 0 : 1)$, and 
    $(1 : 1 : 1)$. Let the equations of $\phi_1$, $\phi_2$, and $\phi_3$ be

    \begin{align*}
        \phi_1  & = ayz + bxz + cxy = 0, \\
        \phi_2  & = a'yz + b'xz + c'xy = 0, \\
        \phi_3  & = a''yz + b''xz + c''xy = 0. 
    \end{align*}

    Assume, for the sake of contradiction, that $\phi_i$ is not in the pencil of $\phi_j$ and $\phi_k$, for some distinct $i, j, k \in \lbrace 1, 2, 3 \rbrace$. Then the matrix in the following system of equations has rank $3$: 

    \begin{equation*}
        \begin{bmatrix}
            a & b & c  \\
         a'&  b'& c' \\
        a''&  b''& c'' \\
        \end{bmatrix} \begin{bmatrix}
            yz \\ 
            xz \\
            xy
        \end{bmatrix} = 0. 
    \end{equation*}

    Hence, the system has only the trivial solution: $yz = 0$, $xz = 0$, $xy = 0$. Therefore, at least two of $x$, $y$, and $z$ must be zero which contradicts the fact that $(1 : 1 : 1)$ is on $\phi_1$, $\phi_2$, and $\phi_3$. 
\end{proof}

\subsection{Background on covering arrays} \label{subsec:ca}

We start this section by defining orthogonal arrays and covering arrays, and investigate some connections between finite geometry and these combinatorial designs, where $G_c^{\ell}$ in \cref{Gengen} plays a central role. 

An \emph{orthogonal array}, denoted $\OA_{\lambda}(N; t, k, v)$, is an $N \times k$ array over an alphabet with $v$ symbols with the property that for any $N \times t$ subarray, each $t$-tuple from the alphabet occurs exactly $\lambda = \frac{N}{v^t}$ times as a row.
Here, $\lambda$ is the \emph{index}, and $t$ is the \emph{strength} of the orthogonal array.
If $\lambda =1$, we remove $\lambda$ from the notation and denote $\OA(N; t, k, v)$. 

In an array over an alphabet of $v$ symbols, a $t$-set of column indices $ \lbrace c_1, \ldots, c_t \rbrace$ is {\emph{$\lambda$-covered}} if each $t$-tuple of the alphabet occurs at least $\lambda$ times as a row of the subarray indexed by $c_1, \ldots, c_t$; we simply use \emph{covered} to denote $1$-covered. 

A \emph{covering array}, denoted by $\CA_{\lambda}(N; t, k, v)$, is an $N \times k$ array over an alphabet with $v$ symbols with the property that any $t$-set of columns is $\lambda$-covered. Here, $N$ is the \emph{size}, and $t$ is the \emph{strength} of the covering array. If $\lambda =1$, we remove $\lambda$ from the notation and denote $\CA(N; t, k, v)$. The covering array number denoted by $\CAN(t, k, v)$ is the minimum $N$ for which a $\CA(N; t, k, v)$ exists. A major application of covering arrays is to generate software test suites to cover all $t$-way interactions between $k$ components, where each component has $v$ values. The number of tests required to exhaustively detect faults in such a system is equal to $v^k$. However, strength-$t$ covering arrays reduce the number of tests dramatically by testing all possible $v^t$ combinations for every $t$-set of factors using $O(\log k)$ tests for fixed $v$ and $t$ \cite{Colbournsurvey2006, Brett2015asym}. This application justifies the importance of covering arrays. For more information about combinatorial testing, see \cite{Nie2011}.
Determining bounds on $\CAN(t, k, v)$ is widely studied \cite{Colbournsurvey2006, handbookcov, Torressurvey2019}. 
Best-known upper bounds for $\CAN(t, k, v)$ for $2 \leq t \leq 6$, $1 \leq k \leq 10000$, and $2 \leq v \leq 25$ are provided in the \href{https://github.com/ugroempi/CAs/blob/main/ColbournTables.md}{covering array tables} maintained by Colbourn \cite{ctable}.
Colbourn \cite{Colbournsurvey2006} and Torres-Jimenez et al. \cite{Torressurvey2019} provided general surveys about covering arrays, where bounds, constructions, and applications of covering arrays are discussed. 

A covering perfect hash family $\CPHF_{\lambda}(n; k, q, t)$ is an $n \times k$ array of elements from $\mathbb{F}_{q}^{t} \setminus \Vec{0}$ (equivalently the points of $\PG(t - 1, q)$) such that, for each
set $T$ of $t$ columns, there exist at least $\lambda$ rows whose entries in the columns of $T$
are linearly independent. If the vector entries all have a non-zero last coordinate
(equivalently, they are the points of the affine subspace of $\PG(t - 1, q)$) then the
CPHF is a Sherwood covering perfect hash family $\SCPHF_{\lambda}(n; k, q, t)$.
Covering perfect hash families are a compressed way to present certain types of covering arrays.
\begin{thm} [\cite{Colbourn2022Brett}] \label{cphftoCA}
Suppose that $C$ is a $\CPHF_{\lambda}(n; k, q, t)$ and $\lambda' \leq \lambda$. Then there exists a $\CA_{\lambda'}(n(q^t  - 1)+ \lambda'; t, k, q)$; and there exists a $\CA_{\lambda'}(n(q^t  - q)+ \lambda'q; t, k, q)$
if $C$ is an $\SCPHF_{\lambda}(n; k, q, t)$.
\end{thm}

\begin{thm}[{\cite[Theorem 2]{Raaphorst2014}}]\label{rank}
    Let $G_k^1$ be the generator matrix with respect to a degree-$m$ primitive polynomial, and $k = \frac{q^m-1}{q-1}$. Then, the following are equivalent:
    \begin{enumerate}
        \item A set of $s$ columns $C =  \lbrace c_{i_{0}}, \cdots , c_{i_{s-1}} \rbrace $ is not covered in $A(G_k^1)$. 
        \item The set of vectors $ \lbrace \alpha^{i_{0}}, \cdots, \alpha^{i_{s-1}} \rbrace$ is linearly dependent over $\mathbb{F}_q$. 
        
        Furthermore, if $s = m$, the following statement is also equivalent to (1) and (2):  
       \item There is a row $r$ other than the all-zero row of $A(G_k^1)$ such that $r_{i_{0}} = \cdots = r_{i_{m-1}} = 0$.
    \end{enumerate}
\end{thm}

A set $D = \lbrace d_1, d_2, \ldots, d_k \rbrace$ of $k \geq 2$ distinct residues modulo $v$ is called a $(v, k, \lambda)$-\emph{difference set} if for every $d \not\equiv 0 \ (\bmod \ v)$ there are exactly $\lambda$ ordered pairs $(d_i, d_j)$ with $d_i, d_j \in D$ such that $d_i - d_j \equiv d \ (\bmod \ v)$. Developing a difference set in $\mathbb{Z}_v$, by taking the translates $D + a = \lbrace d+ a: d \in D \rbrace $ for all $a \in \mathbb{Z}_v$, yields a $2$-$(v, k, \lambda)$ design.

\begin{thm}[{Raaphorst et al., \cite[Theorem 3]{Raaphorst2014}}]\label{bibd}
  	Let 
	$k = \frac{q^m - 1}{q-1}, z= \frac{q^{m-1} - 1}{q-1}, \lambda = \frac{q^{m-2} - 1}{q-1}$.
	Let $M$ be $A(G_k^1)$ with the all-zero row removed.
	Then each row of $M$ has exactly $z$ zeros, and the set:
	\begin{equation}
		\mathcal{B} = \lbrace \lbrace a_1, \ldots, a_{z-1} \rbrace : M_{i, a_1} = \cdots = M_{i, a_{z-1}} = 0 , \  0 \leq i < q^m -2 \rbrace
	\end{equation}
	is the set of blocks of a $2$-$(k, z, \lambda)$ design. Moreover, any block in $\mathcal{B}$ is a $(k, z, \lambda)$-difference set.
\end{thm}

  By \cref{rank} and \cref{bibd}, each block of the  $2$-$(\frac{q^m - 1}{q-1}, \frac{q^{m-1} - 1}{q-1},\frac{q^{m-2} - 1}{q-1})$ design $(\mathbb{Z}_{\frac{q^m -1}{q-1}}, \mathcal{B})$ corresponds to a hyperplane in $\PG(m-1, q)$. 

Let $S$ be an $m \times t$ submatrix $S$ of $G_k^1$ for $1 \leq t \leq m$, and $k = \frac{q^m -1}{q-1}$. Suppose $\rank(S) = r$, for $1 \leq r \leq t$. If $ r \neq t$, then only $q^r$ distinct $t$-tuples are covered in $A(S)$. Otherwise, all distinct $t$-tuples in $\mathbb{F}_q^t$ are covered by some row in $A(S)$. If $t = m$ and $\rank(S) \neq t$ then, by \cref{rank}, the column indices of $S$ are contained in a block of the $2$-$(\frac{q^m - 1}{q-1}, \frac{q^{m-1} - 1}{q-1},\frac{q^{m-2} - 1}{q-1})$ design and the corresponding points in $\PG(m-1, q)$ are in the hyperplane associated with the block. 
Since no two columns of $G_k^{\ell}$ are multiples of each other, the array $A(G_k^{\ell})$ is an $\OA_{q^{m-2}}(q^m; 2, k, q)$ which is also a $\CA_{q^{m-2}}(q^m; 2, k, q)$. We can obtain a strength-$3$ covering array by using $G^{q+1}_{q^2+1}$ for $m=4$ in \cref{Gengen}. 
\begin{cor} \label{3-ca}
     Let $(\mathcal{M}, \mathcal{C})$ be the M{\"o}bius plane corresponding to $G^{q+1}_{q^2+1}$, defined in \cref{3design}. 
     The array $A(G^{q+1}_{q^2+1})$ for $m=4$ in \cref{Gengen} is an $\OA_{q}(q^4; 3, q^2 +1, q) = \CA_{q}(q^4; 3, q^2 +1, q)$, and any $4$-set of columns not covered corresponds to a $4$-set of points that are in a unique block of $\mathcal{C}$.
\end{cor} 

\begin{proof}
    \cref{genovoid} implies any $3$ distinct columns of $G^{q+1}_{q^2+1}$ are linearly independent. Then, \cref{rank} combined with the definition of $(\mathcal{M}, \mathcal{C})$ yields the result. 
\end{proof}

\section{Orthogoval planes and their corresponding covering arrays} \label{sec:orthogoval}
In this section, we review orthogoval planes $\PG(2,q)$ and $\AG(2,q)$ in the literature and review the corresponding constructed covering arrays. 

Two projective (affine) planes with the same point sets are \emph{orthogoval} if the common intersection of any line from each has size at most two \cite{Colbourn2022Brett}. The existence of orthogoval projective planes is proved independently by \cite{Bruck1973, Glynn1978,Hall1947,Dieter1984,Raaphorst2014,Veblen}, and non-independent novel proofs are published in \cite{Baker1994,Colbourn2022Brett,Panario2020}. 

The common idea in these works was to show that for a pair of projective planes with the same pointset, any line in one is an oval in another. They have used different methods using combinatorial design theory and finite geometry. They consider $D$ a $(q^2 + q+1, q+1, 1)$-difference set whose translates are the lines in $\PG(2,q)$, and show that there exists $D'$ obtained from $D$, which corresponds to an oval in $\PG(2,q)$. For instance, Hall \cite{Hall1947} determined that $\frac{1}{2}D = \lbrace x: 2x \in D \rbrace$ is an oval. Bruck \cite{Bruck1973} showed $-D$ is an oval and noted that this gives a pair of planes where lines in one are ovals in the other. Baker et al.~\cite{Baker1994} unified $-D, \frac{1}{2}D$ into a single framework.
Raaphorst et al.~\cite{Raaphorst2014} used a primitive polynomial and its reciprocal to construct strength-$3$ covering arrays and their construction relied on the property that the two projective planes are orthogoval. Panario et al.~\cite{Panario2020} provide a new proof of the existence of a pair of orthogoval planes. They used $G_{q^2 +q +1}^1$ and $G_{q^2 +q +1}^{-1}$ for $\alpha$ not necessarily primitive to represent two projective planes. Colbourn et al.~\cite{Colbourn2022Brett} define the term ``orthogoval'' projective (affine) plane, and provide a new proof for the existence of a pair of orthogoval planes using a Cremona transformation. They extended the definition to an affine plane and showed the existence of a pair of orthogoval affine planes. They also considered sets (of more than two) mutually orthogoval projective (affine) planes.

\cref{geotoca} shows the connection between $s$ mutually orthogoval projective (affine) planes and covering perfect hash families, and hence by \cref{cphftoCA} with covering arrays. 
\begin{thm} [\cite{Colbourn2022Brett}] \label{geotoca}
    Suppose that a set of $s$ mutually orthogoval Desarguesian projective (affine) planes exists. Then there exists a $\CPHF_{s - 1} (s; 3, q^2 + q + 1, q)$ ($\SCPHF_{s - 1}(s; 3, q^2 , q)$). 
\end{thm}

Two additional columns can be added for the covering arrays obtained by $s$ mutually orthogoval Desarguesian affine
planes. 
\begin{thm}[\cite{Colbourn2022Brett}] \label{geotocaplus2}
    Suppose that a set of $s$ mutually orthogoval Desarguesian affine
planes exists. Then there exists a $\CA_{s -1} (sq^3  - q; 3, q^2 + 2, q)$.
\end{thm}

The existence of a pair of orthogoval projective (affine) planes gives \cref{rap}. 

\begin{thm}[\cite{Colbourn2022Brett}] \label{rap}
    There exist a $\CA(2q^3-1; 3, q^2 + q + 1, q)$ for any prime power $q$, and  a $\CA(2q^3-q; 3, q^2 +2, q)$ for even prime power $q$ and $q =3$. 
\end{thm}

The $\CA(2q^3-1; 3, q^2 + q + 1, q)$ in \cref{rap} is obtained by the vertical concatenation of $A(G_{q^2 +q + 1}^1)$ and $A(G_{q^2 +q + 1}^{-1})$.  
It is possible to remove $2(q-1)$ columns of $G_{q^2 +q + 1}^1$ and $G_{q^2 +q + 1}^{-1}$ carefully and create new generators $G_{1}$ and $G_{2}$ respectively,  such that $q$ similar rows can be obtained in $A(G_1)$ and $A(G_{2})$. By removing $q$ redundant rows, we have \cref{rap-tor}. This is equivalent to removing $2(q-1)$ points from a pair of orthogoval projective planes for any prime power $q$ to obtain truncated orthogoval affine planes \cite{torres2018}. 

\begin{thm}[\cite{torres2018}] \label{rap-tor}
    There exists a $\CA(2q^3-q; 3, q^2-q +3, q)$ for any prime power $q$. 
\end{thm}

\cref{tab:sum} shows different geometric objects and corresponding covering arrays. 
\section{Existence of $3$ anti-cocircular truncated M{\"o}bius planes} \label{sec:orthogovoid}
For the rest of the paper, we assume $q$ is an odd prime power and $\alpha$ is a primitive element of $\mathbb{F}_{q^4}$. 

In this section, we introduce three truncated M{\"o}bius planes with the same point set but different circles. Then, we prove they are anti-cocircular. 
We are interested in circles of the M{\"o}bius plane $(\mathcal{M}, \mathcal{C})$ (in \cref{3design}) that contain the column index zero. We will present these circles using the elements of the set $D$ that is given in \cref{dfs}.

\begin{prop} [{\cite[VI, chapter 18]{handbookdif}}] \label{dfs}
    The set $D = \lbrace i \in \mathbb{Z}_{(q^4 -1)/(q-1)}: \Tr(\alpha^i) = 0 \rbrace$ is a $(\frac{q^4 -1}{q-1}, \frac{q^3 -1}{q-1}, \frac{q^2 -1}{q-1})$-difference set over $\mathbb{Z}_{(q^4 -1)/(q-1)}$.  
\end{prop}

The set $D$ corresponds to the tangent plane to the ovoid at $\langle \gamma \rangle = \langle \alpha^{\frac{(q^2 + 1)(q+1)}{2}} \rangle$ with equation $\Tr(X) = 0$. In addition, the hyperplanes in $\PG(3,q)$ correspond to translates of $D$ in $\mathbb{Z}_{\frac{q^4-1}{q-1}}$. Using \cref{rm:omegat}, we consider these hypersurfaces using the coordinatization $\Omega$. 
Let $P \hs \alpha^s x + \alpha^{sq} y + \alpha^{sq^2}z + \alpha^{sq^3}t = 0$ with $\Tr(\alpha^s) = 0$ ($s \in D$). The set  

 \begin{equation} \label{mblock}
        C_{s} = \lbrace i: s + (q+1)i \in D, 0 \leq i < q^2 + 1 \rbrace
    \end{equation} 

    corresponds to points of $\mathcal{O}_c = \mathcal{O}_{\omega}\cap P$ given in \cref{ovoidconic}. If $ \gamma = \alpha^s $, then $P$ is a tangent plane to the ovoid. Therefore, $\lbrace C_s : s \in D \setminus \lbrace \frac{(q^2 + 1)(q+1)}{2} \rbrace \rbrace = \lbrace B \in \mathcal{C} : 0 \in B \rbrace$ where $\mathcal{C}$ is the set of blocks of the M{\"o}bius plane $(\mathcal{M}, \mathcal{C})$. 
In the following, we study some properties of the M{\"o}bius plane $(\mathcal{M}, \mathcal{C})$ and its circles, especially the circles containing the zero element given by \cref{mblock}.

\begin{rem} \label{shift}
    If $0, i, j, k \in C_{x}$,
    then by \cref{rank}, there exist non-zero $a, b, c \in \mathbb{F}_q$ such that $1 +  a\alpha^{i(q+1)}  + b\alpha^{j(q+1)} + c\alpha^{k(q+1)} = 0$. 
     By multiplying $\alpha^{d(q+1)}$ on both sides, we obtain a block $B$ which is a translate of $C_x$ such that $d, i + d, j + d, k + d \in B$ for $ 1 \leq d$.  
\end{rem}

\begin{lem} \label{nocycle}
    Let $C_x$ be a block of $(\mathcal{M}, \mathcal{C})$. It is not possible to have three distinct pairs of points in $C_x$ with the same non-zero difference. 
\end{lem}

\begin{proof}
    Let $x_i, y_i \in C_x$, $1 \leq i \leq 3$ be three distinct pairs of points in $\mathbb{Z}_{q^2+1}$ such that $x_1 - y_1 = x_2 - y_2 = x_3 - y_3 = d$. Then $x_1, x_2, x_3 \in C_x \cap C_x + d$. Since three distinct elements determine a unique block, we have $C_x = C_x + d$. However, the only element in the Singer group that maps a plane to itself is the identity, hence $d = 0$. 
\end{proof}

\begin{prop} \label{centerBlock}
    Let $C_x$ be a block of $(\mathcal{M}, \mathcal{C})$ in \cref{3design}. The following are equivalent:
     \begin{enumerate}
        \item $\frac{q^2 + 1}{2} \in C_x$,
        \item There exists $i \in C_x \setminus \{ 0 , \frac{q^2 +1}{2}\}$ where $-i \in C_x$,
        \item For all $i \in C_x$, $-i \in C_x$. 
     \end{enumerate}
\end{prop}

\begin{proof}
(1) $\Rightarrow$ (2): Since $0, \frac{q^2 +1}{2} \in C_x$, then $x, x + \frac{q^2 +1}{2}(q+1) \in D$. We have $  \Tr(\alpha^x) = \alpha^x + \alpha^{xq} + \alpha^{xq^2} + \alpha^{xq^3} = 0$ and $ \Tr \left (\alpha^{x + \frac{q^2 +1}{2}(q+1)} \right ) = \alpha^{\frac{q^2 +1}{2}(q+1)}\left (\alpha^x - \alpha^{xq} + \alpha^{xq^2} - \alpha^{xq^3} \right ) = 0$, and hence $ \alpha^x - \alpha^{xq} + \alpha^{xq^2} - \alpha^{xq^3} = 0$. Adding the first and the third equations, we obtain $2(\alpha^{x} + \alpha^{xq^2}) = 0$ that gives $\alpha^{x} = - \alpha^{xq^2}$.
Let $i \in C_x$. Then $x + i(q+1) \in D$. So, $ \Tr \left (  \alpha^{x + i(q + 1)} \right ) = 0$. 
Let $\beta = \alpha^{(q+1)(q^2 +1)} \in \mathbb{F}_q$. We have $\beta^{-i} \alpha^{i(q+1)} = \alpha^{-iq^2(q+1)}$ where powers are modulo $q^4 - 1$. Multiplying $ \Tr \left (  \alpha^{x + i(q + 1)} \right ) = 0$ by $\beta^{-i}$ yields 

\begin{equation}
    \beta^{-i}\Tr \left (  \alpha^{x + i(q + 1)} \right ) = \Tr \left ( \beta^{-i} \alpha^{x + i(q + 1)} \right ) = \Tr \left ( \alpha^x \alpha^{-iq^2(q + 1)} \right ) = 0. 
\end{equation}

Substituting $\alpha^{x}$ by $- \alpha^{xq^2}$, we have 

\begin{equation}
    \Tr \left (- \alpha^{xq^2} \alpha^{-iq^2(q + 1)} \right ) = - \Tr \left ( (\alpha^{x-i(q + 1)})^{q^2} \right ) = - \Tr \left ( \alpha^{x-i(q + 1)} \right ) = 0.
\end{equation}
Thus $x - i(q+1) \in D$, and $-i \in C_x$.

(2) $\Rightarrow$ (3):   Let $i \in C_x$ such that $-i \in C_x$. Let $j \in C_x \setminus \{0,  i, -i\}$. Since $0, i, -i, j \in C_x$, by \cref{rank}, then 
$\lbrace \alpha^x, \alpha^{x + i(q+1)}, \alpha^{x-i(q+1)},  \alpha^{x +j(q+1)} \rbrace$ is linearly dependent (but no subset of three vectors is linearly dependent), so there exist non-zero  $a, b, c \in \mathbb{F}_q$ such that 
   \begin{equation} \label{eq:one}
       1 + a\alpha^{i(q+1)} + b\alpha^{-i(q+1)} + c\alpha^{j(q+1)} = 0.
   \end{equation} 
    By raising both sides of the \cref{eq:one} to the power of $q^2$, and the fact that for all $ 0 \leq \ell < q^2 + 1$, $\alpha^{\ell(q+1)q^2} = \beta ^{\ell} \alpha^{-\ell(q+1)}$, we have 

   \begin{equation} \label{eq:two}
       1 + (a \beta^{i})\alpha^{-i(q+1)} + (b \beta^{-i})\alpha^{i(q+1)} + (c \beta^{j})\alpha^{-j(q+1)}  = 0. 
   \end{equation}

   Since $\beta \in \mathbb{F}_q$, \cref{eq:two} shows that there exist a block $B$ of the $3$-$(q^2+1, q+1, 1)$ design $(\mathcal{M}, \mathcal{C})$ such that $0, i, -i, -j \in B$. Since $0, i, -i$ determine a unique block $C_x$, then $0, i, -i, j, -j \in C_x$.  

(3) $\Rightarrow$ (1): 
Assume that $C_x$ is such that if $i \in C_x$, then $-i \in C_x$. Since $|C_x|$ is even, there exists an even number of $i \in C_x$ such that $i = -i$. There are only two elements in $\mathbb{Z}_{q^2 + 1}$ where $j = -j$, namely $0$ and $\frac{q^2 + 1}{2}$. Since $0 \in C_x$, then $\frac{q^2 + 1}{2} \in C_x$. 
\end{proof}

\begin{prop} \label{nodoubleCx}
Suppose $\frac{q^2 + 1}{2} \in C_x$. Then 
\begin{enumerate}
    \item if $i \in C_x \setminus \{ 0, \frac{q^2 + 1}{2}\}$, then $2i \not \in C_x$. \label{no1}
    \item if $k \in C_x \setminus \{ 0, \frac{q^2 + 1}{2}\}$, then $k + \frac{q^2 + 1}{2} \not \in C_x$. \label{no2}
    \item if $i \in C_x \setminus \{ 0, \frac{q^2 + 1}{2}\}$, $k \in C_x \setminus \{ 0, \frac{q^2 + 1}{2}, i, -i\}$ then $2k + i \not \in C_x$. \label{no3}
\end{enumerate}
\end{prop}

\begin{proof}

By assumption, $0, \frac{q^2 + 1}{2} \in C_x$. Then by \cref{centerBlock}, for any $\ell \in C_x$, $-\ell \in C_x$. Each part is proven by using  \cref{nocycle}, which prevents the existence of three distinct pairs of points in $C_x$ with the same difference.

    (1) If $i \in C_x \setminus \{ 0, \frac{q^2 + 1}{2}\}$, then we must have $2i \not \in C_x$, otherwise the following pairs would have the same difference $i$: $(0, -i), (i, 0), (2i, i)$. 
    (2) Similarly, if $k \in C_x \setminus \{ 0, \frac{q^2 + 1}{2}\}$, then $k + \frac{q^2 + 1}{2} \not \in C_x$, otherwise these pairs of points would all have difference $k$: $(k, 0), (0, -k), (k + \frac{q^2 + 1}{2}, \frac{q^2 + 1}{2})$. 
    (3) Similarly, if $i \in C_x \setminus \{ 0, \frac{q^2 + 1}{2}\}$ and $k \in C_x \setminus \{ 0, \frac{q^2 + 1}{2}, i, -i\}$, then $2k + i \not \in C_x$, otherwise we would get three pair of points with difference $i + k$: $(i, -k), (k, -i), (2k +i, k)$. 
\end{proof}

\begin{cor} \label{nodoubleCx12}
    There exists no block $B$ of the $3$-$(q^2 +1, q+1, 1)$ design $(\mathcal{M}, \mathcal{C})$ (given in \cref{3design}) such that $\left \lbrace i, j, i + \frac{q^2 +1}{2}, j + \frac{q^2 +1}{2} \right \rbrace \subseteq B$. In addition, there exists no block containing  $\{ 0, i, j, 2i, 2j \}$.
\end{cor}

\begin{proof}

If $\{i, j, i + \frac{q^2 + 1}{2}, j + \frac{q^2 + 1}{2}\} \subseteq B$, then there exists $x$ such that $\{ 0, j - i, \frac{q^2 + 1}{2}, (j - i) + \frac{q^2 + 1}{2} \} \subseteq C_x$; but this contradicts the second part of \cref{nodoubleCx} by taking $k = j-i$. If $\{ 0, i, j, 2i, 2j \} \subseteq B$, then $\{ -i, 0, j-i, i, 2j-i \} \subseteq C_x$ for some $x$. By \cref{centerBlock}, $\frac{q^2 + 1}{2} \in C_x$, but this contradicts \cref{nodoubleCx} part 3, by taking $k = j-i$ implies $2k + i = 2j -i \not \in C_x$. 
\end{proof}

\begin{con} \label{truncated}
    Let $\mathcal{M}^{(e)} = \lbrace i : i \equiv  0 \ (\mathrm{mod} \ 2), 0 \leq i < q^2 +1 \rbrace$. The set $\mathcal{M}^{(e)}$ is a subset of $\mathcal{M}$ associated to the M{\"o}bius plane $(\mathcal{M}, \mathcal{C})$ given in \cref{3design}. Let $D$ be the associated difference set given in \cref{dfs}. We define three truncated M{\"o}bius planes on the same set of points $\mathcal{M}^{(e)}$. The truncation removes odd points from $\mathcal{M}$ and from the corresponding blocks. Denote for any $C \in \mathcal{C}$, $2C = \lbrace 2i : i \in C \rbrace$, and $(C \cap \mathcal{M}^{(e)})/2 = \lbrace i/2 : i \in C \cap \mathcal{M}^{(e)} , i \equiv 0 \ (\mathrm{mod} \  4) \rbrace \cup \lbrace (i + q^2 +1)/2 : i \in C \cap \mathcal{M}^{(e)}, i \equiv 2 \ (\mathrm{mod} \ 4) \rbrace$.

	\begin{enumerate} 
		\item $M_1 = (\mathcal{M}^{(e)}, \lbrace C \cap \mathcal{M}^{(e)}: C \in \mathcal{C} \rbrace)$. Circles containing zero can be expressed as follows: for any $x \in D$, $C_{x,1} = C_x \cap \mathcal{M}^{(e)} = \lbrace i : x + (q+1)i \in D, i \equiv 0 \ (\mathrm{mod} \ 2),  0 \leq i < q^2 +1  \rbrace$.

		\item $M_2 = (\mathcal{M}^{(e)}, \lbrace (C \cap \mathcal{M}^{(e)})/2: C \in \mathcal{C} \rbrace)$.
		Circles containing zero can be expressed as follows: for any $x \in D$, 
		$C_{x,2} = (C_{x} \cap \mathcal{M}^{(e)}) /2 = \lbrace i : x + 2(q+1)i \in D, i \equiv 0 \ (\mathrm{mod} \ 2),  0 \leq i < q^2 +1  \rbrace$.
		
		\item $M_{1/2} = (\mathcal{M}^{(e)}, \lbrace 2C \cap \mathcal{M}^{(e)}: C \in \mathcal{C} \rbrace)$.
		Circles containing zero can be expressed as follows: for any $x \in D$, 
		$C_{x,1/2} = 2C_x \cap \mathcal{M}^{(e)} = \lbrace i : x + \frac{1}{2}(q+1)i \in D, i \equiv 0 \ (\mathrm{mod} \ 2),  0 \leq i < q^2 +1  \rbrace$.
	\end{enumerate}

\end{con}

\begin{rem}
	It is important to note that $G_{q^2 +1}^{q+1}$ is in correspondence to $(\mathcal{M}, \mathcal{C})$ and 
$G^{q+1}_{\frac{q^2 +1}{2}}$, $G^{2(q+1)}_{\frac{q^2 +1}{2}}$, and $G^{4(q+1)}_{\frac{q^2 +1}{2}}$ with columns labeled with even integers $0, 2, \ldots, q^2 -1$ are in correspondence with $M_{1/2}$, $M_1$, and $M_2$, respectively. Since each of $G^{q+1}_{\frac{q^2 +1}{2}}$, $G^{2(q+1)}_{\frac{q^2 +1}{2}}$, and $G^{4(q+1)}_{\frac{q^2 +1}{2}}$ is a subarray of $G_{q^2 +1}^{q+1}$, having half of its columns, $M_{1/2}$, $M_1$, and $M_2$ are each a truncated M{\"o}bius plane. 
See \cref{tableMob} and \cref{ex:circlesqfive} for an example for $q =5$.
\end{rem}

\begin{figure}[h] 
	
	\ \\ %\ \\
	\centering
	\resizebox{\textwidth}{!}{
		\begin{tabular}{ccc}
			&  \begin{tabularx}{\textwidth}{XXXXXXXXXXXXX} 
				0 & 2 & 4 & 6 & 8 &10 &12 &14 &16 &18 &20 &22 & 24\\
			\end{tabularx}\\
			\\
			$G^{q+1}_{\frac{q^2+1}{2}} = $  &  
			\begin{tabularx}{\textwidth}{|X|X|X|X|X|X|X|X|X|X|X|X|X|} \hline
				$\beta^0$ & $\beta^1$ & $\beta^2$ &$\beta^3$ &$\beta^4$ &$\beta^5$ &$\beta^6$ &$\beta^7$ &$\beta^8$ &$\beta^9$& $\beta^{10}$ &$\beta^{11}$ & $\beta^{12}$\\ \hline 
				
			\end{tabularx}
			& \ \ $M_{1/2}$
			\\
			
			\\
			$G^{2(q+1)}_{\frac{q^2+1}{2}} = $  & 
			\begin{tabularx}{\textwidth}{|X|X|X|X|X|X|X|X|X|X|X|X|X|} \hline
				$\beta^0$ & $\beta^2$ & $\beta^4$ &$\beta^6$ &$\beta^8$ &$\beta^{10}$ &$\beta^{12}$ &$\beta^{14}$ &$\beta^{16}$ &$\beta^{18}$& $\beta^{20}$ &$\beta^{22}$ & $\beta^{24}$\\ \hline 
			\end{tabularx}
			& \ \ $M_{1}$\\
			\\
			$G^{4(q+1)}_{\frac{q^2+1}{2}} =$ &         
			\begin{tabularx}{\textwidth}{|X|X|X|X|X|X|X|X|X|X|X|X|X|}\hline
				$\beta^0$ & $\beta^4$ & $\beta^8$ &$\beta^{12}$ &$\beta^{16}$ &$\beta^{20}$ &$\beta^{24}$ &$\beta^{2}$ &$\beta^{6}$ &$\beta^{10}$& $\beta^{14}$ &$\beta^{18}$ & $\beta^{22}$\\ \hline 
			\end{tabularx}
			& \ \ $M_{2}$\\
			
	\end{tabular}}
    \caption{Generator matrices  for truncated M\"{o}bius planes for $q=5$ ($\beta=\alpha^{q+1}$)} \label{tableMob}
\end{figure}

\begin{table}[h!]
\centering
%\resizebox{\textwidth}{!}{	
	\scalebox{0.8}{
		\setlength{\tabcolsep}{3pt}
		\renewcommand{\arraystretch}{1.1}
		$\begin{array}{|c|l|l|l|l|}
			\hline 
			x \in D \setminus \{78\} & C_x & C_{x,1/2} = 2C_x \cap \mathcal{M}^{(e)} & C_{x,1} = C_x \cap \mathcal{M}^{(e)} & C_{x,2} = (C_x \cap \mathcal{M}^{(e)})/2 \\ \hline
1 & \{0, 2, 4, 9, 15, 21\} & \{0, 4, 8, 18\} & \{0, 2, 4\} & \{0, 2, 14\} \\
5 & \{0, 1, 10, 19, 20, 23\} & \{0, 2, 20\} & \{0, 10, 20\} & \{0, 10, 18\} \\
8 & \{0, 6, 8, 9, 23, 24\} & \{0, 12, 16, 18\} & \{0, 6, 8, 24\} & \{0, 4, 12, 16\} \\
11 & \{0, 9, 18, 19, 22, 25\} & \{0, 18\} & \{0, 18, 22\} & \{0, 22, 24\} \\
13 & \{0, 2, 7, 13, 19, 24\} & \{0, 4, 14\} & \{0, 2, 24\} & \{0, 12, 14\} \\
25 & \{0, 5, 11, 17, 22, 24\} & \{0, 10, 22\} & \{0, 22, 24\} & \{0, 12, 24\} \\
39 & \{0, 8, 12, 13, 14, 18\} & \{0, 16, 24\} & \{0, 8, 12, 14, 18\} & \{0, 4, 6, 20, 22\} \\
40 & \{0, 4, 11, 14, 16, 19\} & \{0, 8, 22\} & \{0, 4, 14, 16\} & \{0, 2, 8, 20\} \\
44 & \{0, 2, 3, 17, 18, 20\} & \{0, 4, 6\} & \{0, 2, 18, 20\} & \{0, 10, 14, 22\} \\
55 & \{0, 6, 12, 17, 19, 21\} & \{0, 12, 24\} & \{0, 6, 12\} & \{0, 6, 16\} \\
56 & \{0, 1, 15, 16, 18, 24\} & \{0, 2\} & \{0, 16, 18, 24\} & \{0, 8, 12, 22\} \\
62 & \{0, 14, 15, 17, 23, 25\} & \{0\} & \{0, 14\} & \{0, 20\} \\
64 & \{0, 7, 10, 12, 15, 22\} & \{0, 14, 20, 24\} & \{0, 10, 12, 22\} & \{0, 6, 18, 24\} \\
65 & \{0, 9, 10, 13, 16, 17\} & \{0, 18, 20\} & \{0, 10, 16\} & \{0, 8, 18\} \\

87 & \{0, 4, 5, 6, 10, 18\} & \{0, 8, 10, 12, 20\} & \{0, 4, 6, 10, 18\} & \{0, 2, 16, 18, 22\} \\
91 & \{0, 6, 11, 13, 15, 20\} & \{0, 12, 22\} & \{0, 6, 20\} & \{0, 10, 16\} \\
106 & \{0, 3, 5, 8, 15, 19\} & \{0, 6, 10, 16\} & \{0, 8\} & \{0, 4\} \\
111 & \{0, 1, 2, 6, 14, 22\} & \{0, 2, 4, 12\} & \{0, 2, 6, 14, 22\} & \{0, 14, 16, 20, 24\} \\
117 & \{0, 1, 5, 13, 21, 25\} & \{0, 2, 10\} & \{0\} & \{0\} \\
119 & \{0, 1, 4, 7, 8, 17\} & \{0, 2, 8, 14, 16\} & \{0, 4, 8\} & \{0, 2, 4\} \\
123 & \{0, 4, 12, 20, 24, 25\} & \{0, 8, 24\} & \{0, 4, 12, 20, 24\} & \{0, 2, 6, 10, 12\} \\
124 & \{0, 2, 5, 12, 16, 23\} & \{0, 4, 10, 24\} & \{0, 2, 12, 16\} & \{0, 6, 8, 14\} \\
125 & \{0, 3, 6, 7, 16, 25\} & \{0, 6, 12, 14\} & \{0, 6, 16\} & \{0, 8, 16\} \\
127 & \{0, 5, 7, 9, 14, 20\} & \{0, 10, 14, 18\} & \{0, 14, 20\} & \{0, 10, 20\} \\
136 & \{0, 3, 10, 14, 21, 24\} & \{0, 6, 20\} & \{0, 10, 14, 24\} & \{0, 12, 18, 20\} \\
143 & \{0, 3, 4, 13, 22, 23\} & \{0, 6, 8\} & \{0, 4, 22\} & \{0, 2, 24\} \\
146 & \{0, 1, 3, 9, 11, 12\} & \{0, 2, 6, 18, 22, 24\} & \{0, 12\} & \{0, 6\} \\
147 & \{0, 8, 16, 20, 21, 22\} & \{0, 16\} & \{0, 8, 16, 20, 22\} & \{0, 4, 8, 10, 24\} \\
152 & \{0, 2, 8, 10, 11, 25\} & \{0, 4, 16, 20, 22\} & \{0, 2, 8, 10\} & \{0, 4, 14, 18\} \\
154 & \{0, 7, 11, 18, 21, 23\} & \{0, 14, 22\} & \{0, 18\} & \{0, 22\} \\
\hline 
\end{array}$
}

\

\resizebox{\textwidth}{!}{$D = \{ 1, 5, 8, 11, 13, 25, 39, 40, 44, 55, 56, 62, 64, 65, 78, 87, 91, 106, 111, 117, 119, 123, 124, 125, 127, 136, 143, 146, 147, 152, 154 \}$}
\caption{List of elements of the set $D$ constructed in \cref{dfs} and circles containing zero of $(\mathcal{M}, \mathcal{C})$, $M_{1/2}$, $M_{1}$, and $M_{2}$, for $q = 5$.} \label{ex:circlesqfive}
\end{table}

Each block of the $2$-$(q^3 + q^2 + q + 1, q^2 + q +1, q+1)$ design in \cref{bibd} for $m = 4$ corresponds to a plane in $\PG(3,q)$, and each $C_{x}$ for $x \in D$ corresponds to the intersection of a plane in $\PG(3,q)$ and the ovoid $\mathcal{O}$ in \cref{genovoid}. Our goal is to connect $C_{x}$ with hypersurfaces in $\PG(3,q)$ which correspond to blocks of $M_2$, and $M_{1/2}$. We will show the existence of $3$ truncated M{\"o}bius planes with the same point sets where the common intersection of any choice of blocks, one from each plane, has at most size three. \cref{tracelemma} is the tool to connect blocks of the M{\"o}bius planes with hypersurfaces in $\PG(3,q)$, and can be considered a generalization of \cite[Theorem 3.2]{Baker1994} from $\PG(2,q)$ to $\PG(3,q)$. 
\begin{prop} \label{tracelemma}
Let $\frac{m}{n} \in \mathbb{Q}$ in lowest terms ($\gcd(m,n) =1$). Suppose $0, i', j', k' \in C_s$ for some $s \in D$, and $i, j, k \in \mathbb{Z}_{(q^2 +1)(q+1)}$ such that $mi = ni'$, $mj = nj'$, and $mk = nk'$.  Then $ \alpha^{\ell(q+1)}$, $\ell = 0, i, j, k$, satisfy $\Tr(\alpha^s X^{\frac{m}{n}}) = 0$ or equivalently $\Omega(\langle \alpha^{\ell(q+1)} \rangle)$ satisfy  $\alpha^s x^{\frac{m}{n}} + \alpha^{sq} y^{\frac{m}{n}} + \alpha^{sq^2} z^{\frac{m}{n}} + \alpha^{sq^3} t^{\frac{m}{n}} = 0$. 
\end{prop}

\begin{proof}
Since $0, i', j', k' \in C_s$, then $\langle \alpha^{\frac{m}{n}\ell(q+1)} \rangle$, $\ell = 0, i, j, k$, satisfy $\Tr(\alpha^s X) = 0$. Therefore, $\langle \alpha^{\ell(q+1)} \rangle$, $\ell = 0, i, j, k$, satisfy $\Tr(\alpha^s X^{\frac{m}{n}}) = 0$. 
\end{proof}

In the following theorems, we study the intersection size of hypersurfaces obtained by blocks of the M{\"o}bius plane $(\mathcal{M}, \mathcal{C})$. \cref{max3} is the main theorem which leads to the existence of an $\ACTM(3,q)$ and $\CA(3q^4-2; 4, \frac{q^2 + 1}{2}, q)$. We give a few other results before proving the main theorem.

\begin{prop} \label{neqc}
    Let $P  \hs \alpha^u x + \alpha^{uq} y + \alpha^{uq^2} z + \alpha^{uq^3} t = 0$, $Q \hs \alpha^v x^2 + \alpha^{vq} y^2 + \alpha^{vq^2} z^2 + \alpha^{vq^3} t^2 = 0$, and $\mathcal{O}_{\omega}\hs xz-yt = 0$ which are hypersurfaces in $\PG(3, q^4)$ for some $u, v \in D \setminus \lbrace \frac{(q^2 +1)(q+1)}{2} \rbrace$. If $u = v$, then $|P \cap Q \cap \mathcal{O}| \leq 2$. 
\end{prop}

\begin{proof}
    Let $\Omega(\langle \alpha^{\ell(q+1)} \rangle) \in P \cap Q \cap \mathcal{O}$ for $\ell = 0, i, j$. If $u = v$, then $\Omega(\langle \alpha^{2\ell(q+1)} \rangle) \in P$ for $\ell = 0, i, j$. This implies that $\left \{ 0, i, j, 2i, 2j \right \} \subseteq C_u$, which is impossible by \cref{nodoubleCx12}. 
\end{proof}

\cref{lemBH} is established via a result by Bruen and Hirschfeld \cite{Hirschfeld1988} and is used in \cref{max4}. 

\begin{lem} \label{lemBH}
Let $Q$ and $\mathcal{O}$ be two non-singular quadrics in $\PG(3,q)$. If every quadric in the pencil of $Q$ and $\mathcal{O}$ is non-singular, then $Q \cap \mathcal{O}$ does not contain a conic.
\end{lem}

\begin{proof}
    Bruen and Hirschfeld \cite{Hirschfeld1988} classify all pencils of quadrics by their common intersections and numbers of singular quadrics in the pencil. By looking at the intersection of quadrics in the cases with no singular quadrics, none of them contain a conic. 
\end{proof}

\begin{thm} \label{max4}
    Let $P  \hs \Tr(\alpha^u X) = 0$, $Q \hs \Tr(\alpha^v X^2) = 0$, and $\mathcal{O} \hs \Tr(\gamma X^{q^2 + 1}) = 0$ be hypersurfaces in $\PG(3, q^4)$ for some $u, v \in D \setminus \lbrace \frac{(q^2 +1)(q+1)}{2} \rbrace$. Then $|P \cap Q \cap \mathcal{O}| \leq 4$. 
    
\end{thm}

\begin{proof}
Recall that $\gamma = \alpha^{\frac{(q^2+1)(q+1)}{2}}$. First, we prove that $Q- \ell \mathcal{O}$ is non-singular for all $\ell \in \mathbb{F}_q^{*}$. Members of $Q - \ell \mathcal{O}$ satisfy $\Tr(\alpha^v X^2 - \ell \gamma X^{q^2 + 1}) = 0$, which is the quadratic form of the bilinear form $B(X, Y) = \Tr((\alpha^v X - \ell \gamma X^{q^2}) Y)$. If $\alpha^vX- \ell \gamma X^{q^2} = 0$, then $X = 0$ or $X^{q^2-1} = \frac{\alpha^v}{\ell \gamma}$. We now show the second equation is not possible. Raising both sides of $X^{q^2-1} = \frac{\alpha^v}{\ell \gamma}$ by $q^2 + 1$ and using $\ell \in \mathbb{F}_q^{*}, \gamma \in \mathbb{F}_{q^2}, \gamma^2 \in \mathbb{F}_q$ gives $\alpha^{v(q^2 + 1)} = \ell^2 \gamma^2 \in \mathbb{F}_q$. Thus, $q^4-1 \mid v(q^2 + 1)(q -1)$ which leads to $v = g(q+1)$ for some integer $g$. So, $\alpha^v \in \mathcal{O}$, which is impossible, since $D$ corresponds to a tangent plane at $\langle \gamma \rangle$. Hence, $X = 0$.  It means that for any $\overline{X} \neq 0$, the linear form $L_{\overline{X}}(Y) = B(\overline{X},Y)$ is not identically zero. Thus $B(X, Y)$ is a non-degenerate bilinear form and the quadratic form $Q - \ell \mathcal{O}$ is non-singular for all $\ell$. 
By \cref{lemBH}, $Q \cap \mathcal{O}$ does not contain a conic.  Let $Q_c = Q \cap P$ and $\mathcal{O}_c = \mathcal{O} \cap P$, which are conics on $P$. Since $Q \cap \mathcal{O}$ does not contain a conic, then $Q_c$ and $\mathcal{O}_c$ are two distinct conics. Since five points determine a unique conic, then $|Q_c \cap \mathcal{O}_c| \leq 4$, which implies $|Q \cap \mathcal{O} \cap P| \leq 4$. 
\end{proof}

\begin{rem} \label{max4forroot}
Let $P  \hs \alpha^u x + \alpha^{uq} y + \alpha^{uq^2} z + \alpha^{uq^3} t = 0$, $H \hs \alpha^v \sqrt{x} + \alpha^{vq} \sqrt{y} + \alpha^{vq^2} \sqrt{z} + \alpha^{vq^3} \sqrt{t} = 0$, and $\mathcal{O}_{\omega}\hs xz-yt = 0$ in $\PG(3, q^4)$ for some $u, v \in D \setminus \lbrace \frac{(q^2 +1)(q+1)}{2} \rbrace$.
    Let $\mathcal{O}^{(e)} = \lbrace \langle \alpha^{2i(q+1)} \rangle \in \mathcal{O}_{\omega}: 0 \leq i < \frac{q^2 +1}{2} \rbrace$. Let $p = \Omega(\langle \alpha^{2i(q+1)} \rangle)$ such that $p \in \mathcal{O}^{(e)} \cap P$ and $p \in \mathcal{O}^{(e)} \cap H$ for some $0 \leq i < \frac{q^2 +1}{2}$. Consider the point $p' = \Omega( \langle \alpha^{i(q+1)} \rangle)$. It is clear that $p' \in \mathcal{O}_{\omega}\cap P'$ and $p' \in \mathcal{O}_{\omega}\cap H'$ where $P'  \hs \alpha^u x^2 + \alpha^{uq} y^2 + \alpha^{uq^2} z^2 + \alpha^{uq^3} t^2 = 0$ and $H'  \hs \alpha^v x + \alpha^{vq} y + \alpha^{vq^2} z + \alpha^{vq^3} t = 0$. Therefore, we have $| \mathcal{O}^{(e)} \cap P| \leq |\mathcal{O}_{\omega}\cap P'|$, and $|\mathcal{O}^{(e)} \cap H| \leq |\mathcal{O}_{\omega}\cap H'|$. Thus by \cref{max4}, $|P \cap H \cap \mathcal{O}^{(e)}| \leq 4$. 
    \end{rem}

   Note that $\mathcal{O}^{(e)}$ corresponds to $\mathcal{M}^{(e)}$ in \cref{truncated}. In simple terms, squaring the points of $\mathcal{O}^{(e)}$ is a bijection that doubles the degrees of hypersurfaces but does not change their intersection size. In what follows, we will be showing that the intersection size is even smaller when a fourth particular hypersurface is considered.

\begin{thm}\label{max2}
    Let $b = \alpha^u$, $c = \alpha^v$ for some $u, v \in D \setminus \lbrace \frac{(q^2 +1)(q+1)}{2} \rbrace$. Let $\mathcal{O}_{\omega}\hs xz-yt = 0$, $P \hs b x + b^{q} y + b^{q^2} z + b^{q^3} t = 0$, and $Q \hs c x^2 + c^{q} y^2 + c^{q^2} z^2 + c^{q^3} t^2 = 0 $ which are hypersurfaces in $\PG(3,q^4)$. Let $\mathcal{O}^{(e)} = \lbrace \Omega(\langle \alpha^{2i(q+1)}) \rangle \in \mathcal{O}_{\omega}: 0 \leq i < \frac{q^2 +1}{2} \rbrace$. If $b^2c^{q^2} = b^{2q^2}c$, then $| \mathcal{O}_{\omega}\cap P \cap Q | \leq 3$, and $| \mathcal{O}^{(e)} \cap P \cap Q | \leq 2$.
\end{thm}

\begin{proof}
    First note that by raising both sides of the equality $b^2c^{q^2} = b^{2q^2}c$ to the power $q$, we have $b^{2q}c^{q^3} = b^{2q^3}c^q$. We use several times in the proof that any powers of $b = \alpha^u$ are invertible. By using $P$ and substituting $t = -b^{-q^3}(b x + b^{q} y + b^{q^2} z)$ in the equations of $\mathcal{O}$ and $Q$, we obtain the equations of the two conics in the plane $P$, which are $\mathcal{O}_c$ as described in \cref{ovoidconic}, and $Q_c$: 

   \begin{multline} \label{quadricconic}
        Q_c \hs \left (b^2c^{q^3} + b^{2q^3}c \right )x^2 + \left (b^{2q}c^{q^3} + b^{2q^3}c^{q} \right ) y^2 + \left (b^{2q^3}c^{q^2} + b^{2q^2}c^{q^3}\right ) z^2 \\ + \left (2b^{q^2 + q}c^{q^3} \right ) yz +  \left (2b^{q^2 + 1}c^{q^3} \right ) xz + \left (2b^{q + 1}c^{q^3} \right ) xy = 0, 
    \end{multline}

Note that $\mathcal{O}_c = \mathcal{O}_{\omega}\cap P$, and $Q_c = Q \cap P$. Using $\mathcal{O}_c$, we substitute $y^2 = \frac{-b^{q^2} yz  - b^{q^3} xz - b xy}{b^q}$ in the equation of $Q_c$, and using the fact that $b^{2q}c^{q^3} = b^{2q^3}c^q$, and $\left ( b^{2q^3}c^{q^2}+b^{2q^2}c^{q^3} \right ) = \left (b^2c^{q^3} + b^{2q^3}c \right )^{q^3}$ we have $\Gamma$:

    \begin{equation} \label{gamma2}
        \Gamma = \left (b^2c^{q^3} + b^{2q^3}c \right )x^2 + \left (b^2c^{q^3} + b^{2q^3}c \right )^{q^3} z^2  +  \left (2c^{q^3}\left (b^{q^2 +1} - b^{q^3 + q} \right ) \right) xz, 
    \end{equation}

Points in $\mathcal{O}_{\omega}\cap P \cap Q$ are of the form $(\alpha^{i(q+1)} : \alpha^{iq(q+1)} : \alpha^{iq^2(q+1)} : \alpha^{iq^3(q+1)})$. We want to determine how many different such points exist. We replace $(x: y: z :t) = (\alpha^{i(q+1)} : \alpha^{iq(q+1)} : \alpha^{iq^2(q+1)} : \alpha^{iq^3(q+1)})$ in $\Gamma = 0$ and divide it by $\alpha^{2i(q+1)}$ to get:

    \begin{equation} \label{eqwithi}
         \left (b^2c^{q^3} + b^{2q^3}c \right ) + \left (b^2c^{q^3} + b^{2q^3}c \right )^{q^3} \left (\alpha^{2i(q^2 -1)(q+1)} \right )  +  \left (2c^{q^3} \left (b^{q^2 +1} - b^{q^3 + q}\right ) \right )  \alpha^{i(q^2 -1)(q+1)} = 0, 
    \end{equation}

     Thus $\alpha^{i(q^2 -1)(q+1)}$ must satisfy:  

     \begin{equation} \label{degree2eq}
          \left (b^2c^{q^3} + b^{2q^3}c \right )^{q^3} X^2  +  \left (2c^{q^3} \left (b^{q^2 +1} - b^{q^3 + q}\right ) \right )  X + \left (b^2c^{q^3} + b^{2q^3}c \right ) = 0, 
    \end{equation}
     which is an equation of degree 2, and has at most two roots. We claim that each root $\bar{x}$ of \cref{degree2eq} gives  two possible solutions $i, i + \frac{q^2 +1}{2} \in \mathbb{Z}_{q^2 +1}$ of \cref{eqwithi}. To see this, suppose $\bar{x}$ is a solution of \cref{degree2eq} such that there exists $i \neq j \in \mathbb{Z}_{q^2 +1}$ where $\bar{x} = \alpha^{i(q^2 -1)(q+1)} = \alpha^{j(q^2 -1)(q+1)}$. So, $\alpha^{(i -j)(q^2 -1)(q+1)} = 1$ which implies that $i - j = k(\frac{q^2 +1}{2})$ for some integer $k$. Since $i, j \in \mathbb{Z}_{q^2 +1}$, then $i - j = 0$ or $j = i + \frac{q^2 +1}{2}$. 
        Suppose \cref{degree2eq} has two solutions $\bar{x}$ and $\bar{y}$ that corresponds to $i$, $i + \frac{q^2 +1}{2}$, $j$, and $j + \frac{q^2 +1}{2}$. Note that since $\frac{q^2 +1}{2}$ is odd and exactly two of the four are even, then $| \mathcal{O}^{(e)} \cap P \cap Q | \leq 2$. 
     It is clear that $i = 0$ gives a solution of \cref{degree2eq}, since $(1: 1: 1 : 1)$ is on $P$, $Q$, and $\mathcal{O}_{\omega}$. Thus, $\alpha^{\ell(q^2 -1)(q+1)}$ for $\ell = 0, \frac{q^2 +1}{2}, j, j + \frac{q^2 +1}{2}$ are four possible solutions of \cref{eqwithi}. We claim that at most three of them can give solutions of \cref{eqwithi}. By contradiction, suppose \cref{eqwithi} has four solutions. So, $\Omega(\langle \alpha^{\ell(q+1)} \rangle) \in P$ for $\ell = 0, \frac{q^2 +1}{2}, j, j + \frac{q^2 +1}{2}$. Then, there exists a block $B$ in $3$-$(q^2 +1, q+1, 1)$ design in \cref{3design} such that $\left \{ 0, \frac{q^2 +1}{2}, j , j + \frac{q^2 +1}{2} \right \} \subseteq B$, which is impossible by the second part of \cref{nodoubleCx}. Therefore, $| \mathcal{O}_{\omega}\cap P \cap Q | \leq 3$. 
\end{proof}

\begin{thm} \label{max2psi}
    Let $\mathcal{O}_{\omega}\hs xz-yt = 0 $, $P \hs b x + b^{q} y + b^{q^2} z + b^{q^3} t = 0$, and 
    \begin{center}
         $\Psi \hs \left ( a \sqrt{x} + a^{q} \sqrt{y} + a^{q^2} \sqrt{z} + a^{q^3} \sqrt{t} \right ) \left ( a \sqrt{x} - a^{q} \sqrt{y} + a^{q^2} \sqrt{z} - a^{q^3} \sqrt{t} \right ) = 0$, 
    \end{center}
    in $\PG(3,q^4)$, where $a = \alpha^s$ and $b = \alpha^u$ for some $s, u \in D \setminus \lbrace \frac{(q^2 +1)(q+1)}{2} \rbrace$. Let $\mathcal{O}^{(e)} = \lbrace \Omega(\langle \alpha^{2i(q+1)} \rangle) \in \mathcal{O}_{\omega}: 0 \leq i < \frac{q^2 +1}{2} \rbrace$. If $a^2b^{q^2} = a^{2q^2}b$, then $ | \mathcal{O}^{(e)} \cap P \cap \Psi | \leq 2$.
\end{thm}

\begin{proof}
    By substituting $t = -b^{-q^3}(bx + b^qy + b^{q^2}z)$ in the equation of $\Psi$, using $xz = yt$, and $a^2b^{q^2} = a^{2q^2}b$, we have 

    \begin{equation}
        \Gamma = \left (  a^2b^{q^3} + a^{2q^3}b \right )x +  \left (  a^{2q^2}b^{q^3} + a^{2q^3}b^{q^2} \right )z + 2b^{q^3}\left ( a^{q^2 +1} - a^{q^3 + q} \right ) \sqrt{xz} = 0. 
    \end{equation}

    Points in $\mathcal{O}^{(e)} \cap P \cap \Psi$ are the form $(\alpha^{2i(q+1)} : \alpha^{2iq(q+1)} : \alpha^{2iq^2(q+1)} : \alpha^{2iq^3(q+1)})$ for $0 \leq i < \frac{q^2 +1}{2}$. We want to determine how many different such points exist. Since $\mathcal{O}^{(e)} \cap P \cap \Psi \hs \Gamma = 0$, we replace $(x:y:z:t) = (\alpha^{2i(q+1)} : \alpha^{2iq(q+1)} : \alpha^{2iq^2(q+1)} : \alpha^{2iq^3(q+1)})$ and after dividing by $\alpha^{2i(q+1)}$, we have

    \begin{equation}
          \left (  a^2b^{q^3} + a^{2q^3}b \right ) +  \left (  a^{2q^2}b^{q^3} + a^{2q^3}b^{q^2} \right )\alpha^{2i(q^2 - 1)(q+1)} + 2b^{q^3}\left ( a^{q^2 +1} - a^{q^3 + q} \right ) \alpha^{i(q^2 -1)(q+1)} = 0. 
    \end{equation}

     Thus $\alpha^{i(q^2 -1)(q+1)}$ must satisfy:  

     \begin{equation} \label{degree2eq2}
          \left (  a^{2q^2}b^{q^3} + a^{2q^3}b^{q^2} \right )X^2 + 2b^{q^3}\left ( a^{q^2 +1} - a^{q^3 + q} \right ) X +  \left (  a^2b^{q^3} + a^{2q^3}b \right ) = 0. 
    \end{equation}

     which is an equation of degree 2, and has at most two roots. 
     We claim that each solution $\bar{x}$ of \cref{degree2eq2} determines a unique $i \in \mathbb{Z}_{\frac{q^2 +1}{2}}$. To see this, suppose $\bar{x}$ is a solution of \cref{degree2eq} such that there exists $i \neq j \in \mathbb{Z}_{\frac{q^2 +1}{2}}$ where $\bar{x} = \alpha^{i(q^2 -1)(q+1)} = \alpha^{j(q^2 -1)(q+1)}$. So, $\alpha^{(i -j)(q^2 -1)(q+1)} = 1$ which implies that $i - j = k(\frac{q^2 +1}{2})$ for some integer $k$. Since $i, j \in \mathbb{Z}_{\frac{q^2 +1}{2}}$, then $i - j = 0$, which implies that $i = j$. Therefore, $| \mathcal{O}^{(e)} \cap P \cap \Psi | \leq 2$. 
\end{proof}

\begin{rem}\label{max2forroot}
 Let $b = \alpha^u$, $a = \alpha^{s}$ for some $u, s \in D \setminus \lbrace \frac{(q^2 +1)(q+1)}{2} \rbrace$. Let $\mathcal{O}_{\omega}\hs xz-yt = 0$, $P \hs b x + b^{q} y + b^{q^2} z + b^{q^3} t = 0$, and $H \hs a \sqrt{x} + a^{q} \sqrt{y} + a^{q^2} \sqrt{z} + a^{q^3} \sqrt{t}  = 0 $ in $\PG(3,q^4)$.  Let $\mathcal{O}^{(e)} = \lbrace \Omega(\langle \alpha^{2i(q+1)} \rangle) \in \mathcal{O}_{\omega}: 0 \leq i < \frac{q^2 +1}{2} \rbrace$. If $a^2b^{q^2} = a^{2q^2}b$, since $H \subseteq \Psi$, by \cref{max2psi}, then $| \mathcal{O}^{(e)} \cap P \cap H | \leq 2$. 
\end{rem}

\begin{thm} \label{max3}
    Let $q$ be an odd prime power and $\alpha$ be a primitive element in $\mathbb{F}_{q^4}$. Let $\mathcal{O}_{\omega}\hs xz-yt = 0 $, $P \hs \alpha^{u} x + \alpha^{uq} y + \alpha^{uq^2} z + \alpha^{uq^3} t = 0$, $Q \hs \alpha^{v} x^2 + \alpha^{vq} y^2 + \alpha^{vq^2} z^2 + \alpha^{vq^3} t^2 = 0$, and 
    \begin{center}
        $\Psi \hs \left ( \alpha^s \sqrt{x} + \alpha^{sq} \sqrt{y} + \alpha^{sq^2} \sqrt{z} + \alpha^{sq^3} \sqrt{t} \right ) \left ( \alpha^s \sqrt{x} - \alpha^{sq} \sqrt{y} + \alpha^{sq^2} \sqrt{z} - \alpha^{sq^3} \sqrt{t} \right ) = 0$, 
    \end{center}
     in $\PG(3,q^4)$, where $s, u, v \in D \setminus \lbrace \frac{(q^2 +1)(q+1)}{2} \rbrace$. Let $\mathcal{O}^{(e)} = \lbrace \Omega(\langle \alpha^{2i(q+1)} \rangle) \in \mathcal{O}_{\omega}: 0 \leq i < \frac{q^2 +1}{2} \rbrace$. Then $ | \mathcal{O}^{(e)} \cap P \cap \Psi \cap Q | \leq 3$. 
\end{thm}

\begin{proof}
Let $a = \alpha^s$, $b = \alpha^u$, and $c = \alpha^v$. We can assume, without loss of generality that $|\mathcal{O}^{(e)} \cap P \cap \Psi| > 2$, thus by \cref{max2psi}, $a^2b^{q^2} \neq a^{2q^2}b$, otherwise we would be done. 

Suppose $p = (x:y:z:t) \in \mathcal{O}^{(e)} \cap \Psi $. We have

\begin{equation}
\begin{split}
    \left (a\sqrt{x} + a^{q}\sqrt{y} + a^{q^2}\sqrt{z} + a^{q^3} \sqrt{t} \right )\left (a\sqrt{x} - a^{q}\sqrt{y} + a^{q^2}\sqrt{z} - a^{q^3} \sqrt{t} \right ) & = 0,  \\
     a^2x + a^{2q^2}z - a^{2q}y - a^{2q^3}t  & =  2a^{q^3 + q}\sqrt{yt} - 2a^{q^2 +1}\sqrt{xz}, \\
     \left (a^2x + a^{2q^2}z - a^{2q}y - a^{2q^3}t \right )^2  & =  \left (2a^{q^3 + q}\sqrt{yt} - 2a^{q^2 +1}\sqrt{xz} \right )^2, \\
    \end{split}
\end{equation}
\begin{multline} \label{eq:QO1111}
    a^{4}x^2 + a^{4q}y^2 + a^{4q^2}z^2 + a^{4q^3} t^2 \\  -2a^2a^{2q}xy + 2a^2a^{2q^2}xz - 2a^2a^{2q^3}xt - 2a^{2q}a^{2q^2}yz + 2a^{2q}a^{2q^3}yt - 2a^{2q^2}a^{2q^3}zt  = \\
    4a^2a^{2q^2} xz + 4a^{2q}a^{2q^3} yt - 8aa^{q}a^{q^2}a^{q^3}\sqrt{xyzt}.
\end{multline}

Since $xz = yt$, then  $p \in \Phi$, where  
 \begin{multline} \label{eq:phia}
        \Phi \hs  a^{4}x^2 + a^{4q}y^2 + a^{4q^2}z^2 + a^{4q^3} t^2 \\ - 2a^2a^{2q}xy - 2a^2a^{2q^2}xz - 2a^2a^{2q^3}xt - 2a^{2q}a^{2q^2}yz - 2a^{2q}a^{2q^3}yt - 2a^{2q^2}a^{2q^3}zt  \\ + 8aa^{q}a^{q^2}a^{q^3}xz = 0.
    \end{multline}

 This shows that $\mathcal{O}^{(e)} \cap P \cap Q \cap \Psi \subseteq \mathcal{O}^{(e)} \cap P \cap Q \cap \Phi$. Next, we will show that $|P \cap \mathcal{O}^{(e)} \cap \Phi \cap Q| \leq 3$, which implies $|P \cap \mathcal{O}^{(e)} \cap \Psi \cap Q| \leq 3$. 
    Assume that the intersection of $P$, $Q$, $\Phi$, and $\mathcal{O}^{(e)}$ has at least four points.
    Using the equation of $P$, we substitute $t = -b^{-q^3}(b x + b^{q} y + b^{q^2} z)$ in the equation of $\mathcal{O}$, $Q$, and $\Phi$, and obtain the equation of three conics in the plane $P$, namely $\mathcal{O}_c$, $Q_c$ as described in \cref{ovoidconic} and \cref{quadricconic}, and $\Phi_c$: 

\begin{multline} \label{phiconic}
    \Phi_c \hs \left (a^2b^{q^3} + a^{2q^3}b \right )^2 x^2 + \left (a^{2q}b^{q^3} + a^{2q^3}b^q \right )^2 y^2 + \left (a^{2q^2}b^{q^3} + a^{2q^3}b^{q^2} \right )^2 z^2 \\ 
    + \left (2 \left (a^{2q}b^{q^2} + a^{2q^2}b^q \right )a^{2q^3}b^{q^3} - 2a^{2q^2 + 2q}b^{2q^3} + 2a^{4q^3}b^{q^2 + q} \right )yz \\ 
     + \left (2 \left (a^2b^{q^2} + a^{2q^2}b \right )a^{2q^3}b^{q^3} + 2 \left (4a^{q^3 + q^2 + q + 1} - a^{2q^2 + 2} \right )b^{2q^3} + 2a^{4q^3}b^{q^2 + 1} \right ) xz \\ 
     + \left (2 \left (a^2b^q + a^{2q}b \right )a^{2q^3}b^{q^3} - 2a^{2q + 2}b^{2q^3} + 2a^{4q^3}b^{q + 1} \right ) xy = 0. 
\end{multline}

Since we assume there are four points in $\mathcal{O}^{(e)} \cap Q \cap \Phi \cap P$, then by \cref{onepencil}, $\Phi_c$ should be in the pencil of $\mathcal{O}_c$ and $Q_c$. Consider the $3 \times 6$ matrix $M'$ where each row corresponds to the coefficients of $x^2$, $y^2$, $z^2$, $yz$, $xz$, $xy$ for  $\mathcal{O}_c$, $Q_c$, and $\Phi_c$, respectively: 

\begin{equation}
    M' = \begin{bmatrix}
      0 &  b^{q}& 0 & b^{q^2} & b^{q^3} & b \\ 
(b^2c^{q^3} + b^{2q^3}c)& (b^{2q}c^{q^3} + b^{2q^3}c^{q}) & (b^{2q^2}c^{q^3} + b^{2q^3}c^{q^2}) & (2b^{q^2 + q}c^{q^3}) &  (2b^{q^2 + 1}c^{q^3}) & (2b^{q + 1}c^{q^3}) \\

      (a^2b^{q^3} + a^{2q^3}b)^2 & (a^{2q}b^{q^3} + a^{2q^3}b^q)^2 & (a^{2q^2}b^{q^3} + a^{2q^3}b^{q^2})^2 & A & B & C \\
    \end{bmatrix},
\end{equation}
    where $A$, $B$, and $C$ are the coefficient of $yz$, $xz$, and $xy$ in $\Phi_c$, respectively: 

    \begin{equation} \label{eqA}
        A = \left (2 \left (a^{2q}b^{q^2} + a^{2q^2}b^q \right )a^{2q^3}b^{q^3} - 2a^{2q^2 + 2q}b^{2q^3} + 2a^{4q^3}b^{q^2 + q} \right ),
    \end{equation}

    \begin{equation}
        B = \left (2 \left (a^2b^{q^2} + a^{2q^2}b \right )a^{2q^3}b^{q^3} + 2 \left (4a^{q^3 + q^2 + q + 1} - a^{2q^2 + 2} \right )b^{2q^3} + 2a^{4q^3}b^{q^2 + 1} \right ),
    \end{equation}

    \begin{equation} \label{eqC}
        C = \left (2 \left (a^2b^q + a^{2q}b \right )a^{2q^3}b^{q^3} - 2a^{2q + 2}b^{2q^3} + 2a^{4q^3}b^{q + 1} \right ).
    \end{equation}
    
    We consider the last three columns of the above matrix and call it $M$. We calculate $\det(M)$ and show that $\det(M) \neq 0$.

    \begin{equation} \label{eq:detr}
       \det(M) = 2bb^{q+1}c^{q^3} \begin{vmatrix}
            b^{q^2-1} & b^{q^3-1} & 1 \\ 
            b^{q^2 -1} &  b^{q^2 -q} & 1 \\
            A & B & C \\
        \end{vmatrix} = 2bb^{q+1}c^{q^3} \begin{vmatrix}
            b^{q^2-1} & b^{q^3-1} & 1 \\ 
            0 &  b^{q^2 -q} - b^{q^3-1} & 0 \\
            A & B & C \\
        \end{vmatrix}.
    \end{equation}

    Note that $b^{q^2 -q} - b^{q^3-1}$ cannot be equal to zero: if $b^{q^2 -q} = b^{q^3-1}$, then $b^{(q-1)(q^2+1)} = 1$ and we have $q^4 -1 \ | \ u(q^2 +1)(q-1)$ which leads to $q+1 \ | \ u$. Hence $u = k(q+1)$, for some integer $k$. Since $u \in D$, then the only element in $D$ which is multiple of $q+1$ is $\frac{(q^2 + 1)(q+1)}{2}$. Hence, $u = \frac{(q^2 + 1)(q+1)}{2}$ which contradicts the choice of $u$ in the assumption. So, we have

    \begin{equation}
         \det(M) = 2bb^{q+1}c^{q^3} \left (b^{q^2 -q} - b^{q^3-1} \right ) \begin{vmatrix}
            b^{q^2-1} & 0 & 1 \\ 
            0 &  1 & 0 \\
            A & 0 & C \\
        \end{vmatrix} =  -2bc^{q^3} \left (b^{q^3 + q} - b^{q^2 +1} \right )\begin{vmatrix}
            b^{q^2-1}  & 1 \\ 
            A  & C \\
        \end{vmatrix}.
    \end{equation}

Since $b$, $c$ are powers of $\alpha$, and $-b^{q+1}\left (b^{q^2 -q} - b^{q^3-1} \right ) =  \left (b^{q^3 + q} - b^{q^2 +1} \right ) \neq 0$, then $\det(M) \neq 0$ if and only if $b \left (b^{q^2 -1}C - A \right ) \neq 0$. 
    For the final step, we calculate $b \left (b^{q^2 -1}C - A \right )$:

    \begin{multline}
     b \left (b^{q^2 -1}C - A \right ) = 2 ( a^{2q^3}b^{q^3}b^qa^2b^{q^2} + a^{2q}a^{2q^3}bb^{q^3}b^{q^2} - a^{2q}b^{2q^3}a^2b^{q^2} + a^{4q^3}b^{q^2}b^{q}b \\ -
        a^{2q}a^{2q^3}bb^{q^3}b^{q^2} - a^{2q^3}b^{q^3}b^{q}a^{2q^2}b + a^{2q}b^{2q^3}a^{2q^2}b - 
        a^{4q^3}b^{q^2}b^{q}b ), 
    \end{multline}
    \begin{equation}
       = 2 \left ( a^{2q^3}b^{q^3}b^{q} \left (a^2b^{q^2} - a^{2q^2}b \right ) - a^{2q}b^{2q^3} \left (a^2b^{q^2} - a^{2q^2}b \right ) \right ), 
    \end{equation}
    \begin{equation} \label{blast}
       = - 2b^{q^3} \left (a^2b^{q^2} - a^{2q^2}b \right ) \left ( a^{2q}b^{q^3} - a^{2q^3}b^q\right ).
    \end{equation}
    \begin{equation}
      = -2b^{q^3} \left (a^2b^{q^2} - a^{2q^2}b \right )^{q+1}
    \end{equation}

Note that in \cref{blast}, we have used the fact that $\left ( a^{2q}b^{q^3} - a^{2q^3}b^q\right ) = \left (a^2b^{q^2} - a^{2q^2}b \right )^q$. 
Hence, 

 \begin{equation}
        \det(M) = 4b^{q^3}c^{q^3} \left (b^{q^3 + q} - b^{q^2 +1} \right ) \left (a^2b^{q^2} - a^{2q^2}b \right )^{q+1}. 
    \end{equation}
    
    By \cref{max2psi}, and using that $|\mathcal{O}^{(e)} \cap P \cap \Psi| > 2$, we know $\left (a^2b^{q^2} - a^{2q^2}b \right ) \neq 0$, and since $4b^{q^3}c^{q^3} \neq 0$, we have $\det(M) \neq 0$. 
    Therefore, the rank of $M$ and $M'$ is equal to $3$. This implies $\Phi_c$ is not in the pencil of $\mathcal{O}_c$ and $Q_c$, and we reach the desired contradiction. 
\end{proof}

\begin{cor} \label{max3H}
    Let $\mathcal{O}_{\omega}\hs xz-yt = 0 $, $P \hs \alpha^{u} x + \alpha^{uq} y + \alpha^{uq^2} z + \alpha^{uq^3} t = 0$, $Q \hs \alpha^{v} x^2 + \alpha^{vq} y^2 + \alpha^{vq^2} z^2 + \alpha^{vq^3} t^2 = 0$, and $H \hs \alpha^{s} \sqrt{x} + \alpha^{sq} \sqrt{y} + \alpha^{sq^2} \sqrt{z} + \alpha^{sq^3} \sqrt{t}  = 0$ in $\PG(3,q^4)$, where $s, u, v \in D \setminus \lbrace \frac{(q^2 +1)(q+1)}{2} \rbrace$. Let $\mathcal{O}^{(e)} = \lbrace \Omega(\langle \alpha^{2i(q+1)} \rangle) \in \mathcal{O}_{\omega}: 0 \leq i < \frac{q^2 +1}{2} \rbrace$. Then $ | \mathcal{O}^{(e)} \cap P \cap H \cap Q | \leq 3$. 
\end{cor}

\begin{proof}
    Since $H \subseteq \Psi$, then the result is immediate by \cref{max3}. 
\end{proof}

The results obtained so far on the sizes of the intersections of various hypersurfaces are summarized in \cref{tab:summary}, which provides a general overview of the different cases.

\begin{table}[h]
    \centering
    $\begin{array}{||c|c|c|c|c|c|c|c|c||}
    \hline \hline 
        \mathcal{O}_{\omega} & \mathcal{O}^{(e)} & P & Q & H & \Psi & \text{Condition} & \text{Size of intersection} & \text{Reference}\\
         \hline \hline 
         \checkmark & & \checkmark& \checkmark & & & u = v& | P \cap Q \cap \mathcal{O}_{\omega}| \leq 2  & \cref{neqc}\\
         \hline 
        \checkmark & & \checkmark& \checkmark & & & & | P \cap Q \cap \mathcal{O}_{\omega}| \leq 4  & \cref{max4}\\
         \hline 
         & \checkmark & \checkmark &  & \checkmark & & & | P \cap H \cap \mathcal{O}^{(e)}| \leq 4 & \cref{max4forroot}\\
         \hline 
        \checkmark & & \checkmark & \checkmark & & & \alpha^{2u} \alpha^{vq^2} = \alpha^{2uq^2}\alpha^{v} & | P \cap Q \cap \mathcal{O}_{\omega}| \leq 3 & \cref{max2}\\
         \hline 
          &\checkmark & \checkmark & \checkmark & & & \alpha^{2u} \alpha^{vq^2} = \alpha^{2uq^2}\alpha^{v} & | P \cap Q \cap \mathcal{O}^{(e)}| \leq 2 & \cref{max2}\\
         \hline 
         & \checkmark & \checkmark & & & \checkmark &\alpha^{2s} \alpha^{uq^2} = \alpha^{2sq^2}\alpha^{u} & | P \cap \Psi \cap \mathcal{O}^{(e)}| \leq 2 & \cref{max2psi}\\
         \hline 
         & \checkmark & \checkmark & & \checkmark &  &\alpha^{2s} \alpha^{uq^2} = \alpha^{2sq^2}\alpha^{u} & | P \cap H \cap \mathcal{O}^{(e)}| \leq 2 & \cref{max2forroot}\\
         \hline 
       & \checkmark & \checkmark & \checkmark&  & \checkmark & & | P \cap Q \cap \Psi \cap \mathcal{O}^{(e)}| \leq 3 & \cref{max3}\\
         \hline 
        & \checkmark & \checkmark & \checkmark&  \checkmark&  & & | P \cap Q \cap H \cap \mathcal{O}^{(e)}| \leq 3 & \cref{max3H}\\
         
         \hline \hline 
         \multicolumn{9}{||c||}{\text{\scriptsize{$\mathcal{O}_{\omega} \hs xz - yt = 0$}}} \\
         \multicolumn{9}{||c||}{\text{\scriptsize{$\mathcal{O}^{(e)} = \{\Omega(\langle \alpha^{2i(q+1)} \rangle) \in \mathcal{O}_{\omega}: 0 \leq i < \frac{q^2 +1}{2} \}$}}} \\
         \multicolumn{9}{||c||}{\text{\scriptsize{$P \hs \alpha^{u} x + \alpha^{uq} y + \alpha^{uq^2} z + \alpha^{uq^3} t = 0$}}} \\
         \multicolumn{9}{||c||}{\text{\scriptsize{$Q \hs \alpha^{v} x^2 + \alpha^{vq} y^2 + \alpha^{vq^2} z^2 + \alpha^{vq^3} t^2 = 0$}}} \\
         \multicolumn{9}{||c||}{\text{\scriptsize{$H \hs \alpha^{s} \sqrt{x} + \alpha^{sq} \sqrt{y} + \alpha^{sq^2} \sqrt{z} + \alpha^{sq^3} \sqrt{t}  = 0$}}} \\
         \multicolumn{9}{||c||}{\text{\scriptsize{$\Psi \hs \left ( \alpha^s \sqrt{x} + \alpha^{sq} \sqrt{y} + \alpha^{sq^2} \sqrt{z} + \alpha^{sq^3} \sqrt{t} \right ) \left ( \alpha^s \sqrt{x} - \alpha^{sq} \sqrt{y} + \alpha^{sq^2} \sqrt{z} - \alpha^{sq^3} \sqrt{t} \right ) = 0$}}} \\
         \multicolumn{9}{||c||}{\text{\scriptsize{$s, u, v \in D \setminus \{\frac{(q^2 +1)(q+1)}{2} \}$}}} \\
    \hline\hline 
    \end{array}$
    \caption{Size of intersection of particular hypersurfaces}
    \label{tab:summary}
\end{table}

As a direct consequence of \cref{max3H}, we have the following central result of this paper.  

\begin{thm} \label{orthogovoideven}
	Let $q$ be an odd prime power. Let  $M_{1/2} = (\mathcal{M}^{(e)}, \lbrace 2C \cap \mathcal{M}^{(e)}: C \in \mathcal{C} \rbrace)$, $M_1 = (\mathcal{M}^{(e)}, \lbrace C \cap \mathcal{M}^{(e)}: C \in \mathcal{C} \rbrace)$, and $M_2 = (\mathcal{M}^{(e)}, \lbrace (C \cap \mathcal{M}^{(e)})/2: C \in \mathcal{C} \rbrace)$ be the three truncated M{\"o}bius planes in \cref{truncated}. Then, $M_{1/2}$, $M_1$, and $M_2$ give an $\ACTM(3,q)$. 
\end{thm}
\begin{proof}
	Let $B_{1/2}$, $B_1$, and $B_2$ be circles of $M_{1/2}$, $M_{1}$, and $M_{2}$, respectively. Let $I = B_{1/2} \cap B_1 \cap B_2$. We must prove $|I| \leq 3$. Let $d \in I$ and $I_0 = \{ i -d \ : \ i \in I  \}$; note that $0 \in I_0$. Assume without loss of generality that $|I_0| > 1$. By \cref{shift}, for $i \in B_{\ell}$, $\ell = 1/2, 1, 2$, there exists $x \in D \setminus \{ \frac{(q^2 +1)(q+1)}{2} \}$ such that $\ell i \in C_x$ and $i \in C_{x, \ell}$. Therefore, there exists $s, u, v \in D \setminus \{ \frac{(q^2 +1)(q+1)}{2} \}$ such that $I_0 \subseteq C_{s,1/2} \cap C_{u,1} \cap C_{v,2}$. Let $I_L = \{ \Omega(\langle \alpha^{i(q+1))} \rangle) : i \in I_0 \}$.  By \cref{tracelemma}, 
$I_L \subseteq \mathcal{O}^{(e)} \cap H \cap P \cap Q$, where $H \hs a\sqrt{x} + a^q \sqrt{y} + a^{q^2} \sqrt{z} + a^{q^3} \sqrt{t} = 0$, $P \hs bx + b^q y + b^{q^2} z + b^{q^3}t = 0$, $Q \hs cx^2 + c^q y^2 + c^{q^2} z^2 + c^{q^3}t^2 = 0$, and $a = \alpha^s$, $b = \alpha^u$, and $c = \alpha^v$. 
\cref{max3H} implies $|I_L| \leq 3$, therefore $|I| = |B_{1/2} \cap B_1 \cap B_2| \leq 3$. 
\end{proof}

\begin{lem} \label{transfer}
    Let $P_a \hs ax + a^q y + a^{q^2} z + a^{q^3}t = 0$, $Q_b \hs bx^2 + b^q y^2 + b^{q^2} z^2 + b^{q^3}t^2 = 0$, $R_c \hs cx^4 + c^q y^4 + c^{q^2} z^4 + c^{q^3}t^4 = 0$. Let  
        $\Psi_a \hs \left ( a\sqrt{x} + a^q \sqrt{y} + a^{q^2} \sqrt{z} + a^{q^3} \sqrt{t}          \right ) \left ( a\sqrt{x} - a^q \sqrt{y} + a^{q^2} \sqrt{z} - a^{q^3} \sqrt{t}          \right ) = 0$, $P_b \hs bx + b^q y + b^{q^2} z + b^{q^3}t = 0$, and $Q_c \hs cx^2 + c^q y^2 + c^{q^2} z^2 + c^{q^3}t^2 = 0$, 
    where $a = \alpha^{s}$, $b = \alpha^{u}$, and $c = \alpha^{v}$ for some  $s, u, v \in D \setminus \lbrace \frac{(q^2 +1)(q+1)}{2} \rbrace$. 
    Let $p' = \Omega(\langle \alpha^{(\ell + p_{\ell}\frac{q^2 +1}{2})(q+1)} \rangle)$,
     where $p_{\ell} \in \lbrace 0, 1 \rbrace$, and $p = \Omega(\langle \alpha^{2l(q+1)} \rangle)$. Then the following hold: 

     \begin{enumerate}

         \item if $p' \in P_a$ then $p \in \Psi_a$; \label{it:1}

         \item if   $p' \in Q_b$ then $p \in P_b$; \label{it:2}

         \item  if $p' \in R_c$ then $p \in Q_c$. \label{it:3}
    \end{enumerate}
\end{lem}

\begin{proof}

We first prove~\cref{it:1}. Note that for any $0 \leq \ell < \frac{q^2 +1}{2}$, we have $\Omega(\langle \alpha^{(\ell + \frac{q^2 +1}{2})(q+1)} \rangle) = (\alpha^{\ell (q+1)}: -\alpha^{\ell q(q+1)} : \alpha^{\ell q^2(q+1)} : -\alpha^{\ell q^3(q+1)})$. So, $p' \in P_a$ implies $a \alpha^{\ell (q+1)} + a^q \alpha^{\ell q(q+1)} + a^{q^2} \alpha^{\ell q^2(q+1)} + a^{q^3} \alpha^{\ell q^3(q+1)} = 0$ if $p_\ell = 0$, or $a \alpha^{\ell (q+1)} - a^q \alpha^{\ell q(q+1)} + a^{q^2} \alpha^{\ell q^2(q+1)} - a^{q^3} \alpha^{\ell q^3(q+1)} = 0$ if $p_\ell = 1$. In either case, $p \in \Psi_a$. \cref{it:2} and \cref{it:3} are clear using the corresponding equations.
\end{proof}

\begin{thm} \label{fullovoid}
    Let $q$ be an odd prime power and $\alpha$ be a primitive element in $\mathbb{F}_{q^4}$. Let $\mathcal{O}_{\omega}\hs xz-yt = 0 $, Let $P_a \hs ax + a^q y + a^{q^2} z + a^{q^3}t = 0$, $Q_b \hs bx^2 + b^q y^2 + b^{q^2} z^2 + b^{q^3}t^2 = 0$, $R_c \hs cx^4 + c^q y^4 + c^{q^2} z^4 + c^{q^3}t^4 = 0$ in $\PG(3,q^4)$, where $a = \alpha^{s}$, $b = \alpha^{u}$, and $c = \alpha^{v}$ for some  $s, u, v \in D \setminus \lbrace \frac{(q^2 +1)(q+1)}{2} \rbrace$.  Let $\Vec{p} = (p_0, p_1, \ldots, p_i, \ldots, p_{\frac{q^2 -1}{2}}) \in \lbrace 0, 1 \rbrace^{\frac{q^2 +1}{2}}$, and $\mathcal{O}^{(\Vec{p})} = \lbrace \Omega(\langle \alpha^{(i + p_i\frac{q^2 +1}{2})(q+1)} \rangle) \in \mathcal{O}_{\omega}: 0 \leq i < \frac{q^2 +1}{2} \rbrace$. Then $ | \mathcal{O}^{(\Vec{p})} \cap P_a \cap Q_b \cap R_c | \leq 3$. 
\end{thm}

\begin{proof}
    By contradiction, assume $ | \mathcal{O}^{(\Vec{p})} \cap P_a \cap Q_b \cap R_c |  \geq 4$. Then, there exists 

\begin{equation*}
    F = \left  \lbrace \Omega(\langle \alpha^{(\ell + p_\ell\frac{q^2 +1}{2})(q+1)} \rangle) : \ell \in \lbrace 0, i, j, k  \rbrace \right \rbrace \subseteq \mathcal{O}^{(\Vec{p})} \cap P_a \cap Q_b \cap R_c,
\end{equation*}

such that $\ell = 0, i, j, k$ $(\text{mod }  \frac{q^2 +1}{2})$ are distinct. 
Then,
\begin{equation*}
    F^{(2)} = \left  \lbrace \Omega(\langle \alpha^{2\ell(q+1)} \rangle) : \ell \in \lbrace 0, i, j, k  \rbrace \right \rbrace 
\end{equation*}
has four distinct points. 
Then, by \cref{transfer}, $F^{(2)} \subseteq \mathcal{O}^{(e)} \cap P_b \cap Q_c \cap \Psi_a$, where 
\begin{center}
    $\Psi_a \hs \left ( a\sqrt{x} + a^q \sqrt{y} + a^{q^2} \sqrt{z} + a^{q^3} \sqrt{t}          \right ) \left ( a\sqrt{x} - a^q \sqrt{y} + a^{q^2} \sqrt{z} - a^{q^3} \sqrt{t}          \right ) = 0$, 
\end{center}
$P_b \hs bx + b^q y + b^{q^2} z + b^{q^3}t = 0$, and $Q_c \hs cx^2 + c^q y^2 + c^{q^2} z^2 + c^{q^3}t^2 = 0$. By \cref{max3}, $|\mathcal{O}^{(e)} \cap P_b \cap Q_c \cap \Psi_a| \leq 3$, which is a contradiction with $|F^{(2)}| = 4$.  
\end{proof}

\begin{rem} \label{pshift}
Let $f$ be the natural bijection between $\left \{ 0, 1 \right \}^{\frac{q^2 +1}{2}}$ and the set 
$S = \{ B \subseteq \mathbb{Z}_{q^2 +1} : |B| = \frac{q^2 +1}{2} \text{ and for all } 0 \leq i < \frac{q^2 +1}{2}, \text{either }i \in B \text{ or } i + \frac{q^2 +1}{2} \in B \}$; more specifically, let  $f(\Vec{p}) = \{ i + p_i(\frac{q^2 +1}{2}) : 0 \leq i < \frac{q^2 +1}{2} \}$. The set $S$ is closed under cyclic shifts, that is, if $B \in S$ then $B + \ell \in S$, since the fact that any $i, j \in B$ has $i-j \neq \frac{q^2 +1}{2}$ implies the same fact for $B + \ell$. Therefore, if $B = f(\Vec{p})$, then there exists a $\vec{p'}$ such that $B + \ell = f(\Vec{p'})$. 
\end{rem}
\begin{nota}
    Let $\vec{p} = (p_0, p_1, \ldots, p_{\frac{q^2 -1}{2}}) \in \left \{ 0, 1 \right \}^{\frac{q^2 +1}{2}}$. We denote the $ 4 \times \frac{q^2 +1}{2}$ submatrix of $G^{\ell(q+1)}_{q^2 +1}$ with column indices $\{ i + p_{i}(\frac{q^2 +1}{2}) : 0 \leq i < \frac{q^2 + 1}{2} \}$ by $\left [ G^{\ell(q+1)}_{q^2 +1} \right ]_{\Vec{p}}$. 
\end{nota}

\begin{con} \label{mopvec}
    Let $\vec{p} = (p_0, p_1, \ldots, p_{\frac{q^2 -1}{2}}) \in \left \{ 0, 1 \right \}^{\frac{q^2 +1}{2}}$ and $(\mathcal{M}, \mathcal{C})$ be the M{\"o}bius plane given in \cref{3design}. 
    We construct three truncated M{\"o}bius planes $M^{(\vec{p})}_1$, $M^{(\vec{p})}_2$, and $M^{(\vec{p})}_4$, with the same point set $\mathcal{M}^{(\vec{p})} = \left \{ i + p_i(\frac{q^2 +1}{2}) : 0 \leq i < \frac{q^2 +1}{2} \right \}$, where each corresponds to $\left [ G^{q+1}_{q^2 +1} \right ]_{\Vec{p}}$, $\left [ G^{2(q+1)}_{q^2 +1} \right ]_{\Vec{p}}$, and $\left [ G^{4(q+1)}_{q^2 +1} \right ]_{\Vec{p}}$, respectively. More specifically, 
    \begin{enumerate}
        \item $M_1^{\vec{p}} = (\mathcal{M}^{\Vec{p}}, \{ C \cap \mathcal{M}^{\Vec{p}} : C \in \mathcal{C} \})$. 
        \item $M_2^{\vec{p}} = (\mathcal{M}^{\Vec{p}}, \{ \{ j \in \mathcal{M}^{\Vec{p}} : 2j \in C \} : C \in \mathcal{C} \})$. 
        \item  $M_4^{\vec{p}} = (\mathcal{M}^{\Vec{p}}, \{ \{ j \in \mathcal{M}^{\Vec{p}} : 4j \in C \} : C \in \mathcal{C} \})$.
    \end{enumerate}
    By \cref{3design}, $G_{q^2 +1}^{q+1}$ is in correspondence to $(\mathcal{M}, \mathcal{C})$. Since each $\left [ G^{q+1}_{q^2 +1} \right ]_{\Vec{p}}$, $\left [ G^{2(q+1)}_{q^2 +1} \right ]_{\Vec{p}}$, and $\left [ G^{4(q+1)}_{q^2 +1} \right ]_{\Vec{p}}$ is a subarray of $G_{q^2 +1}^{q+1}$, having half of its columns, then $M^{(\Vec{p})}_{1}$, $M^{(\Vec{p})}_{2}$, $M^{(\Vec{p})}_{4}$ are each a truncated M{\"o}bius plane. 
\end{con}

\begin{thm} \label{orthogovoid} 
     Let $q$ be an odd prime power. $M_1^{(\vec{p})}$, $M_2^{(\vec{p})}$, and $M_4^{(\vec{p})}$ in \cref{mopvec} give an $\ACTM(3,q)$. 
\end{thm}

\begin{proof}

Let $B_{1}$, $B_2$, and $B_4$ be circles of $M_{1}^{\Vec{p}}$, $M_{2}^{\Vec{p}}$, and $M_{4}^{\Vec{p}}$, respectively. Let $I = B_1 \cap B_2 \cap B_4$. We must prove $|I| \leq 3$. Let $d \in I$ and $I_0 = \{ i - d : i \in I \}$; note that $0 \in I_0$. Assume without loss of generality that $|I_0| > 1$. By \cref{shift}, for $i \in B_{\ell}$, $\ell = 1, 2, 4$, there exists $x \in D \setminus \{ \frac{(q^2 +1)(q+1)}{2} \}$ such that $\ell i \in C_x$.
Hence, there exist $C_{s}$, $C_{u}$, and $C_{v}$, for $s, u, v \in D \setminus \{ \frac{(q^2 +1)(q+1)}{2} \}$ such that $i \in C_s$, $2i \in C_u$, and $4i \in C_v$ for any $i \in I_0$. Let $I_L = \{ \Omega( \langle \alpha^{i(q+1))} \rangle) :i \in I_0 \}$. From \cref{pshift}, we note that after the shift by $d$, we have changed $\Vec{p}$ to $\Vec{p'}$. 
Thus, by \cref{tracelemma}, $I_L \subseteq \mathcal{O}^{(\Vec{p'})} \cap P_a \cap Q_b \cap R_c$, where $P_a \hs ax + a^q y + a^{q^2} z + a^{q^3}t = 0$, $Q_b \hs bx^2 + b^q y^2 + b^{q^2} z^2 + b^{q^3}t^2 = 0$, $R_c \hs cx^4 + c^q y^4 + c^{q^2} z^4 + c^{q^3}t^4 = 0$, and $a = \alpha^s$, $b = \alpha^u$, $c = \alpha^v$. \cref{fullovoid} implies $|I_L| \leq 3$, therefore $|I| =|B_{1} \cap B_2 \cap B_4| \leq 3$. 
\end{proof}

\section{Construction of strength-$4$ covering arrays} \label{sec:4ca3layer}
In this section, we prove the existence of a $\CA(3q^4 -2; 4, \frac{q^2 +1}{2}, q)$ in
\cref{4ca3layer}, which can be considered as an extension of \cref{rap} to higher dimension. 

\begin{thm} \label{4ca3layer}
    Let $q$ be an odd prime power. Let $\vec{p} = (p_0, p_1, \ldots, p_{\frac{q^2 -1}{2}}) \in \left \{ 0, 1 \right \}^{\frac{q^2 +1}{2}}$. Let $\alpha$ be a primitive element in $\mathbb{F}_{q^4}$ and let $A_{1} = A \left (\left [ G^{q+1}_{q^2 +1} \right ]_{\Vec{p}} \right )$, $A_2 = A\left (\left [ G^{2(q+1)}_{q^2 +1} \right ]_{\Vec{p}} \right )$, and $A_4 = A\left (\left [ G^{4(q+1)}_{q^2 +1} \right ]_{\Vec{p}} \right )$ be arrays obtained from $\alpha$ in \cref{Gengen}. The vertical concatenation of $A_{1}$, $A_2$, and $A_4$ with two copies of the all-zero rows removed is a $\CA(3q^4 -2; 4, \frac{q^2 + 1}{2}, q)$. 
\end{thm}

\begin{proof}
    Let $M$ be the vertical concatenation of $A_1$, $A_2$, and $A_4$. Consider any four distinct columns $i, j, k, l$ of $M$. Note that $i, j, k, l$ are elements of the truncated M{\"o}bius planes $M_{1}^{\Vec{p}}$, $M_{2}^{\Vec{p}}$, and $M_{4}^{\Vec{p}}$, corresponding to $A_1$, $A_2$, and $A_4$, respectively. By \cref{orthogovoid}, there exists no circles containing $i, j, k, l$ in at least one of these three M{\"o}bius planes, say $M_x^{\Vec{p}}$ for $x \in \{1, 2, 4 \}$. By \cref{3design},  every row of the corresponding $q^4 \times 4$ submatrix of $A_x$ indexed by column indices $i, j, k, l$ has at most three zeros. By \cref{rank}, all distinct $4$-tuples corresponding to column indices  $i, j, k, l$ are covered in $A_x$ and therefore in $M$.  
\end{proof}

In \cref{table:4ca3layer}, we display the size $N_s$ of the $\CA(3q^4 -2; 4, \frac{q^2 +1}{2}, q)$ obtained by \cref{4ca3layer} and the size $N_c$ of the best-known covering arrays with $k = \frac{q^2 +1}{2}$ obtained from \cite{ctable}, for any odd prime power $q \leq 25$. The column indicated by ``Method'' shows the method by which each of the best-known arrays were obtained. The reference for each method is mentioned. The column indicated by ``$N_s - N_c$'' shows the difference between the size of the array obtained by \cref{4ca3layer} and the size of the best-known covering array in \cite{ctable}. We can see that for $q \geq 11$, we have an improvement of $21 - 25$ percent in the size of covering arrays obtained in \cref{4ca3layer}. 

\begin{table}[h]
\centering
    \scalebox{0.9}{
        $\begin{array}{||c||c|c|c|c|c|c||}
        \hline \hline 
            q & k = \frac{q^2 +1}{2}&\text{CAs from \cref{4ca3layer}} & \multicolumn{2}{|c|}{\text{The previously best-known CAs \cite{ctable}} }  &  & \\
            \hline 
             & &N_s = 3q^4 -2 &  N_c & \text{Method}& N_s - N_c &  (N_s - N_c)/N_c\\
            \hline \hline 
            3 & 5&241 & 81 & \text{Derive from strength 5} & 160 & 1.975\\
            \hline
            5 & 13&1873 &  1225 & \text{2-Restricted SCPHF RE (CL)} \cite{Colbourn2020Lanus} &  648 & 0.528\\
            \hline
            7 & 25 &7201 & 6853 &\text{3-Restricted SCPHF RE (CL)} \cite{Colbourn2020Lanus} & 348 & 0.05\\
            \hline
            9 & 41&19681  & 19593 & \text{2-Restricted SCPHF RE (CL)} \cite{Colbourn2020Lanus} & 88 & 0.004\\
            \hline
            11 & 61&\textbf{43921} & 55891& \text{3,3-Restricted SCPHF RE (CL)} \cite{Colbourn2020Lanus} & -11970 & -0.214\\
            \hline
            13 &  85 &\textbf{85681} & 109837 & \text{3,3-Restricted SCPHF RE (CL)} \cite{Colbourn2020Lanus}& -24156 & -0.219\\
            \hline
            17 &  145&\textbf{250561} & 329137& \text{3-Restricted SCPHF RE (CL) }\cite{Colbourn2020Lanus} & -78576 & -0.238\\
            \hline
            19 & 181 & \textbf{390961}& 520543 & \text{2,2-Restricted SCPHF RE (CL)} \cite{Colbourn2020Lanus}    &  -129582 & -0.248\\
            \hline
            23 & 265 &\textbf{839521} & 1119361 & \text{CPHF IPO 4 (WCS)} \cite{Wagner2022}  & -279840 & -0.249\\
            \hline
            25 & 313&\textbf{1171873} & 1562497 & \text{CPHF IPO 4 (WCS)} \cite{Wagner2022} & -390606 & -0.249\\
            \hline \hline 

        \end{array}$}
        \caption{Covering arrays of strength 4 obtained by \cref{4ca3layer} compared with the previously best-known CAs of strength 4 in \cite{ctable} for odd prime power $q \leq 25$. Improvements in size are shown in bold.}
    \label{table:4ca3layer}
    \end{table}
\begin{cor}
    For an odd prime power $q$, a $\CPHF(3; \frac{q^2+1}{2}, q, 4)$ exists where each row is a $\CPHF(1; \frac{q^2+1}{2}, q, 3)$. 
\end{cor}
\begin{proof}
    The proof is immediate by \cref{4ca3layer}. 
\end{proof}

The generator matrices $G^{q+1}_{\frac{q^2 +1}{2}}$, $G^{2(q+1)}_{\frac{q^2 +1}{2}}$, and $G^{4(q+1)}_{\frac{q^2 +1}{2}}$ with respect to the primitive polynomial $f(x) = x^4 + 5x^2 + 4x + 3$ over $\mathbb{F}_7$ is given in \cref{fig:G1G2G4}. Each of $G^{q+1}_{\frac{q^2 +1}{2}}$, $G^{2(q+1)}_{\frac{q^2 +1}{2}}$, and $G^{4(q+1)}_{\frac{q^2 +1}{2}}$ corresponds to a row of a $\CPHF(3; 25, 7, 4)$, and each column vector of these arrays, correspond to an entry in the corresponding column of the $\CPHF(3; 25, 7, 4)$. 
\begin{figure}[h]
    \centering
    \scalebox{0.7}{
    \begin{tabular}{cc}
    $G^{q+1}_{\frac{q^2 +1}{2}} =$ & $\begin{bmatrix}
    \begin{array}{ccccccccccccccccccccccccc}
0 & 5 & 1 & 4 & 2 & 1 & 0 & 5 & 0 & 6 & 0 & 1 & 6 & 1 & 6 & 6 & 0 & 4 & 6 & 3 & 1 & 0 & 6 & 2 & 3 \\
0 & 5 & 6 & 6 & 4 & 5 & 1 & 6 & 2 & 0 & 6 & 1 & 4 & 0 & 5 & 3 & 5 & 5 & 5 & 5 & 1 & 6 & 0 & 6 & 3 \\
0 & 1 & 1 & 1 & 5 & 5 & 3 & 1 & 0 & 1 & 3 & 6 & 3 & 1 & 2 & 0 & 5 & 0 & 4 & 6 & 5 & 2 & 3 & 3 & 0 \\
1 & 4 & 6 & 5 & 0 & 2 & 0 & 3 & 1 & 2 & 5 & 4 & 4 & 0 & 1 & 2 & 3 & 2 & 4 & 5 & 1 & 0 & 6 & 5 & 1
\end{array}
\end{bmatrix}$ \\
    \vspace{0.3cm}
    \\
    $G^{2(q+1)}_{\frac{q^2 +1}{2}} =$ & $\begin{bmatrix}
    \begin{array}{ccccccccccccccccccccccccc}
0 & 1 & 2 & 0 & 0 & 0 & 6 & 6 & 0 & 6 & 1 & 6 & 3 & 3 & 1 & 1 & 1 & 4 & 1 & 2 & 5 & 3 & 5 & 4 & 6 \\
0 & 6 & 4 & 1 & 2 & 6 & 4 & 5 & 5 & 5 & 1 & 0 & 3 & 6 & 6 & 6 & 0 & 3 & 5 & 6 & 0 & 6 & 5 & 3 & 3 \\
0 & 1 & 5 & 3 & 0 & 3 & 3 & 2 & 5 & 4 & 5 & 3 & 0 & 5 & 5 & 2 & 0 & 2 & 0 & 4 & 1 & 3 & 3 & 0 & 6 \\
1 & 6 & 0 & 0 & 1 & 5 & 4 & 1 & 3 & 4 & 1 & 6 & 1 & 2 & 2 & 6 & 4 & 1 & 1 & 5 & 4 & 4 & 6 & 4 & 4
\end{array}
\end{bmatrix}$ \\
    
    \vspace{0.3cm}
    \\
  $G^{4(q+1)}_{\frac{q^2 +1}{2}} =$ & $\begin{bmatrix}
   \begin{array}{ccccccccccccccccccccccccc}
0 & 2 & 0 & 6 & 0 & 1 & 3 & 1 & 1 & 1 & 5 & 5 & 6 & 3 & 0 & 0 & 4 & 4 & 4 & 2 & 3 & 5 & 6 & 2 & 5 \\
0 & 4 & 2 & 4 & 5 & 1 & 3 & 6 & 0 & 5 & 0 & 5 & 3 & 4 & 3 & 4 & 1 & 1 & 0 & 4 & 4 & 2 & 4 & 4 & 2 \\
0 & 5 & 0 & 3 & 5 & 5 & 0 & 5 & 0 & 0 & 1 & 3 & 6 & 3 & 2 & 2 & 6 & 5 & 2 & 1 & 6 & 6 & 5 & 2 & 0 \\
1 & 0 & 1 & 4 & 3 & 1 & 1 & 2 & 4 & 1 & 4 & 6 & 4 & 4 & 0 & 1 & 3 & 5 & 4 & 6 & 4 & 3 & 1 & 5 & 5
\end{array}
\end{bmatrix}$ \\
    
    \end{tabular}}
    \caption{The generator matrices $G^{q+1}_{\frac{q^2 +1}{2}}$, $G^{2(q+1)}_{\frac{q^2 +1}{2}}$, and $G^{4(q+1)}_{\frac{q^2 +1}{2}}$ with respect to the primitive polynomial $f(x) = x^4 + 5x^2 + 4x + 3$ over $\mathbb{F}_7$.}
    \label{fig:G1G2G4}
\end{figure}

\section{A recursive construction of strength-$4$ covering arrays using all points of the ovoid} \label{sec:rec4-ca}

In this section, we construct a strength-$4$ covering array using a recursive method. The main ingredients are the $\CA(q^4; 3, q^2 +1, q)$ in \cref{3-ca} and the $\CA(3q^4-1; 4, \frac{q^2 +1}{2}, 1)$ in \cref{4ca3layer} for an odd prime power $q$.

\begin{nota} \label{added}
    Let $A = (a_{ij}) \in \mathbb{F}_q^{m \times n}$ and let $x \in \mathbb{F}_q$. Denote by $A + x \in \mathbb{F}_q^{m \times n}$ the matrix $B = (b_{ij})$ with $b_{ij} = a_{ij} + x$.
\end{nota}

\begin{thm} \label{4ca-rec}
	Let $q$ be an odd prime power. Suppose there exists a $\CA(N; 3, \frac{q^2+1}{2}, q)$. Then a $\CA(3q^4 + N(q-2); 4, q^2 +1, q)$ exists.  
\end{thm}

\begin{proof}
	Let $A (G^{q+1}_{q^2 +1})$ be the $\CA_{q}(q^4; 3, q^2 +1, q)$ in \cref{3-ca}, and $A(G^{2(q+1)}_{\frac{q^2 +1}{2}})$ and $A(G^{4(q+1)}_{\frac{q^2 +1}{2}})$ be two $\CA_{q}(q^4; 3, \frac{q^2 +1}{2}, q)$, which are subarrays of $A (G^{q+1}_{q^2 +1})$. Note that $ \left [A(G_{q^2 +1}^{2(q+1)})\right ]  =\left [ A(G^{2(q+1)}_{\frac{q^2 +1}{2}}) | A(G^{2(q+1)}_{\frac{q^2 +1}{2}}) \right ]$ and $\left [A(G_{q^2 +1}^{4(q+1)})\right ] = \left [ A(G^{4(q+1)}_{\frac{q^2 +1}{2}}) | A(G^{4(q+1)}_{\frac{q^2 +1}{2}}) \right ]$. Let $R$ be a $\CA(N; 3, \frac{q^2 +1}{2}, q)$. Let $X$ be the $3q^4 + N(q-2) \times (q^2 +1)$ array described via its subarrays in \cref{tab:4ca-rec}, displayed using $q+1$ row blocks labeled from $0$ to $q$. We show $X$ is a $\CA(3q^4 + N(q-2); 4, q^2 +1, q)$.
	
	\begin{table}[h]
		\centering
		$ X = \begin{array}{|c|c|}
			\hline
			\multicolumn{2}{|c|}{A(G^{q+1}_{q^2+1})} \\
			\hline
			A(G^{2(q+1)}_{\frac{q^2 +1}{2}}) & A(G^{2(q+1)}_{\frac{q^2 +1}{2}}) \\
			\hline
			A(G^{4(q+1)}_{\frac{q^2 +1}{2}})& A(G^{4(q+1)}_{\frac{q^2 +1}{2}}) + e^{0} \\
			\hline
			R & R + e^{1} \\
			\hline 
			R & R + e^{2} \\
			\hline 
			
			\vdots & \vdots \\
			\hline
			R & R + e^{q-2} \\
			\hline 
		\end{array}$
		\caption{A $\CA(3q^4 + N(q-2); 4, q^2 +1, q)$ constructed in \cref{4ca-rec}.}
		\label{tab:4ca-rec}
	\end{table}
	
		Every column $c$ of $X$ can be represented as $(x_c, p_c)$ where $c = x_c + p_c(\frac{q^2 +1}{2})$ and $0 \leq x_c < \frac{q^2 +1}{2}$, $p_c \in \left \{ 0, 1\right \}$. Let $\lbrace c_1, c_2, c_3, c_4 \rbrace$ be a set of four distinct column indices of $X$ represented in this way by $\{(i, p_{c_1}), (j, p_{c_2}), (k, p_{c_3}), (\ell, p_{c_4}) \}$, respectively.  We analyze the three possible cases based on the multiset $\{i, j, k, \ell\}$, which by construction has elements of multiplicity at most $2$, since $p_{c_i} \in \{ 0, 1\}$; these cases are listed next:

	\begin{enumerate}
		\item $i, j, k, \ell$ all distinct; \label{case1}

		\item $\ell = i $; $i, j, k$ distinct; \label{case2}

		\item  $i = k, j = \ell$, $0 \leq i < j < \frac{q^2 +1}{2}$. \label{case3}
		
	\end{enumerate}

	\textbf{Case~\ref{case1}:} 
		In this case, we show that the coverage of $c_1, c_2, c_3, c_4$ is guaranteed by the first three row blocks of $X$.
		Since $i, j, k,\ell$ are all distinct, there exists a $\Vec{p} \in \lbrace 0, 1 \rbrace^{\frac{q^2 +1}{2}}$ such that $p_i = p_{c_1}$, $p_j = p_{c_2}$, $p_k = p_{c_3}$, and $p_\ell = p_{c_4}$.   
		By \cref{4ca3layer}, the $4 \times 4$ submatrix corresponding to these column indices has rank equal to $4$ in at least one of $ \left [ G^{q+1}_{q^2 +1} \right ]_{\Vec{p}}$, $ \left [ G^{2(q+1)}_{q^2 +1} \right ]_{\Vec{p}}$, and $\left [ G^{4(q+1)}_{q^2 +1} \right ]_{\Vec{p}}$. This means that the set of columns $S = \{ i, j, k, \ell \}$ is covered is in at least one of  $A \left (\left [ G^{q+1}_{q^2 +1} \right ]_{\Vec{p}} \right )$, $A \left (\left [ G^{2(q+1)}_{q^2 +1} \right ]_{\Vec{p}} \right )$, and $A \left (\left [ G^{4(q+1)}_{q^2 +1} \right ]_{\Vec{p}} \right )$. If $S$ is covered in one of the first two arrays, then so is $\{c_1, c_2, c_3, c_4 \}$ in $X$. In the case that $S$ is only covered in $A \left (\left [ G^{4(q+1)}_{q^2 +1} \right ]_{\Vec{p}} \right )$, we need to verify $\{c_1, c_2, c_3, c_4 \}$ is covered in the third row block of $X$, where in the second half, $e^0 = 1$ has been added to every element. Indeed, since the set of covered tuples in the submatrix of $\left [ A(G^{4(q+1)}_{\frac{q^2 +1}{2}}) | A(G^{4(q+1)}_{\frac{q^2 +1}{2}}) \right ]$ indexed by $\{c_1, c_2, c_3, c_4 \}$ is $\mathbb{F}_{q}^4$, then the set of tuples in the submatrix of $\left [ A(G^{4(q+1)}_{\frac{q^2 +1}{2}}) | A(G^{4(q+1)}_{\frac{q^2 +1}{2}}) + e^0 \right ]$ is
		\[
		\mathbb{F}_q^4 + (p_i, p_j, p_k, p_\ell) = \{ (a, b, c, d) + (p_i, p_j, p_k, p_\ell) : (a, b, c, d) \in \mathbb{F}_q^4 \} = \mathbb{F}_q^4.
		\]
		This means $\{c_1, c_2, c_3, c_4 \}$ is covered in $X$. Therefore, one of the first three row blocks of $X$ guarantees coverage of all distinct four-tuples for $c_1$, $c_2$, $c_3$, $c_4$ in $X$.   
		
		\textbf{Case~\ref{case2}:}
		Let $\mathbb{F}_{q} = \{ x_1, x_2, \ldots, x_q \}$ where $x_1 = 0$, $x_r = e^{r-2}$ for $2 \leq r \leq q$. The set $S_r = \lbrace (a, b + p_{c_2}x_r , c +p_{c_3}x_r, a + x_r): a, b, c \in \mathbb{F}_q, p_{c_2}, p_{c_3} \in \lbrace 0, 1 \rbrace \rbrace$ is the set of four-tuples covered in row blocks $1 \leq r \leq q$ in columns $c_i, c_j, c_k, c_\ell$.  First note that $|S_r| = q^3$, for all $1 \leq r \leq q$. We will show $S_r \cap S_{r'} = \emptyset$, whenever $r \neq r'$. 
		Let $1 \leq r , r^{'} \leq q$ and $(a, b, c, d) \in S_r \cap S_{r'}$. Then $(a, b, c, a + x_r) = (a, b, c, a + x_{r'})$ which implies $r = r'$. Hence, $S_{r} \cap S_{r'} = \emptyset$, whenever $r \neq r'$. Let $S = \bigcup_{1 \leq r \leq q} S_r$. Then $|S| = q^4$, which implies $S = \mathbb{F}_q^4$. Thus, all distinct four-tuples are covered for $c_1$, $c_2$, $c_3$, $c_4$ in $X$. 
		
		\textbf{Case~\ref{case3}:}
		We claim that the $4 \times 4$ submatrix of $G^{q+1}_{q^2 +1}$ in the row block zero and corresponding to column indices $i, j, i + \frac{q^2 +1}{2}, j + \frac{q^2 +1}{2}$ has rank equal to $4$. 
		Indeed, \cref{nodoubleCx12} shows that no blocks contain $i, j, i + \frac{q^2 +1}{2}, j + \frac{q^2 +1}{2}$ in the M{\"o}bius plane from \cref{3design}. Therefore, the corresponding points in the ovoid are not co-planar, and the rank of the $4 \times 4$ submatrix is four. Thus, all distinct four-tuples are covered in the row block zero. 
\end{proof}

\begin{rem}
    The first three row blocks of the array $X$ in \cref{4ca-rec} give the coverage for all column indices $i, j, k, l$ except when exactly three of them are distinct. This shows that the array obtained by the first three row blocks is a strength-$3$ covering array and ``almost'' a strength-$4$ covering array. 
\end{rem}

\begin{rem} \label{tor}
    There exists a $\CA(2q^3-q; 3, \frac{q^2 +1}{2}, q)$ for any prime power $q$. This has the best-known size for $\frac{q^2 +1}{2}$ in \cite{ctable}. This covering array can be obtained by using the first $\frac{q^2 +1}{2}$ columns of the $\CA(2q^3-q; 3, q^2-q +3, q)$ from \cref{rap-tor}. 
\end{rem}

\begin{cor} \label{4ca-recpro}
    For an odd prime power $q$, there exists a $\CA(3q^4 + (q-2)(2q^3-q); 4, q^2 +1, q)$. 
\end{cor}

\begin{proof}
    The result is obtained by using the $\CA(2q^3 -q; 3, \frac{q^2 +1}{2}, q)$ from \cref{tor} in \cref{4ca-rec}. 
\end{proof}

In \cref{tab:4ca-rec-impr}, we compare the size $N_s$ of the strength-$4$ covering array obtained in \cref{4ca-rec}, namely $N_s = 3q^4 + (2q^3 -q)(q-2)$, with $N_c$, the size of the best-known strength-$4$ covering array with $k = q^2 +1$ columns, and $q$ symbols found in \cite{ctable} as shown in \cref{4ca-recpro}. For odd prime powers $q \geq 11$, \cref{4ca-recpro} improves bounds.  

\begin{table}
    \scalebox{0.78}{
        $\begin{array}{||c||c||c||c|c||c|c||}
        \hline \hline 
            q & k = q^2 +1& \text{Obtained by \cref{4ca-recpro}} & \multicolumn{2}{|c|}{\text{The previously best-known CAs \cite{ctable}} }  &  & \\
            \hline 
             & &N_s = 3q^4 + (2q^3-q)(q-2)  &  N_c & \text{Method}& N_s - N_c &  (N_s - N_c)/N_c\\
            \hline \hline 
            3 & 10 & 294 & 159 & \text{CPHF 3-stage (TJ-IM)} \cite{Torres-Jimenez2018cphf3stage} & 135 &0.849 \\
            \hline
            5 & 26 & 2610 & 1865 &	\text{Restricted CPHF Sim Annealing (TJ-IM)}  \cite{ctable} & 745 &0.399  \\
            \hline
            7 & 50& 10598  & 9247	& \text{3-Restricted SCPHF RE (CL)} \cite{Colbourn2020Lanus}& 1351& 0.146 \\
            \hline
            9 & 62& 29826 & 26241 &	\text{CPHF IPO 4 (WCS)} \cite{Wagner2022} &  3585 & 0.136\\
            \hline
            11 &122& \textbf{67782} & 70521 &	\text{3,3-Restricted SCPHF RE (CL)} \cite{Colbourn2020Lanus} & -2739 & -0.038 \\
            \hline
            13 &170& \textbf{133874}& 138385 &	\text{3,3-Restricted SCPHF RE (CL)} \cite{Colbourn2020Lanus} &  -4511 & -0.032 \\
            \hline
            17 &290& \textbf{397698}  & 412369 &	\text{3,2-Restricted SCPHF RE (CL)} \cite{Colbourn2020Lanus} &  -14671 & -0.035 \\
            \hline
            19 &362& \textbf{623846}  & 644347&	\text{3,2-Restricted SCPHF RE (CL)} \cite{Colbourn2020Lanus}  &   -20501 & -0.031 \\
            \hline
            23 &530&\textbf{1350054} & 1398101 &	\text{2,2-Restricted SCPHF RE (CL)} \cite{Colbourn2020Lanus}&  -48047 & -0.034 \\
            \hline
            25 &626&\textbf{1890050} & 1951825 &	\text{2,2-Restricted SCPHF RE (CL)} \cite{Colbourn2020Lanus} &  -61775 & -0.031 \\
            \hline \hline 

        \end{array}$}
        \caption{Covering arrays of strength 4 obtained by \cref{4ca-recpro} compared with the previously best-known CAs of strength 4 in \cite{ctable} for odd prime power $q \leq 25$. Improvements in rows are shown in bold.}
    \label{tab:4ca-rec-impr}
    \end{table}

\section{Conclusion and future work} \label{sec:future}
In this paper, we proved the existence of three truncated M{\"o}bius planes for any odd prime power $q$, where for any choice of blocks from each plane, the common intersection size is at most three. This result is a generalization of the existence of orthogoval planes studied in  \cite{Baker1994,Bruck1973,Colbourn2022Brett,Glynn1978,Hall1947,Dieter1984,Panario2020,Raaphorst2014,Veblen}, and played a key role to build a $\CA(3q^4-2; 4, \frac{q^2 +1}{2}, q)$ in \cref{4ca3layer}. The $\CA(2q^3-1; 3, q^2+q+1, q)$ constructed by Raaphorst et al. \cite{Raaphorst2014}, which is obtained from orthogoval projective planes, improved the size of the best-known arrays by almost 33 percent and are still the best-known covering arrays in \cite{ctable} for their number of columns. Similarly, for odd prime power $q \geq 11$, we obtain an improvement by almost 25 percent in \cref{4ca3layer} on the size of the previously best-known strength-$4$ covering arrays with the same parameters. 
This shows that constructive methods using structures from geometry can make a significant improvement in the size of covering arrays. These two constructions for $t = 3, 4$ strongly suggest continuing this approach on the strength-$5$ covering arrays. There are two examples that are obtained by vertical concatenation of two strength-$4$ covering arrays, which suggests this structure; the first is a $\CA(485; 5, 11, 3)$ \cite{Tzanakis2016}, the second one is a $\CA(6249; 5, 11, 5)$, which we found using experiments. These suggest the construction of  strength-$5$ covering arrays by the vertical concatenation of four strength-$4$ covering arrays for any prime power $q$. Understanding the geometry in these two examples and finding such a construction is our plan for future work. 

We used the $\CA(3q^4 -2; 4, \frac{q^2 + 1}{2}, q)$ in a recursive construction as the main ingredient to double the number of columns and less than double the number of rows, constructing a $\CA(3q^4 + (q-2)(3q^3 -q), 4, q^2 +1, q)$ for an odd prime power $q$ (\cref{4ca-recpro}). For $q \geq 11$, these covering arrays improve the size of the best-known arrays with the same parameters by around $3$ percent. New families of strength-$3$ covering arrays were constructed in \cite{SM2024} by horizontal concatenation of $x$ copies of $\CA(2q^3-1; 3, q^2 + q + 1, 3)$ where $ x \in \lbrace 2, q, q^2, q^2-q+1 \rbrace$ for a prime power $q$. The next step is to leverage our geometric construction with recursive methods using more than two copies of $\CA(3q^4-2; 4, \frac{q^2+1}{2}, q)$. Investigating the possibility of effective approaches to construct strength-$4$ covering arrays using recursive methods is another of our plans for the future. 

The first clue leading us to construct the $\CA(3q^4-2; 4, \frac{q^2+1}{2}, q)$ in \cref{4ca3layer} for an odd prime power $q$ was the existence of the $\CA(511; 4, 17, 4)$ obtained by Tzanakis et al. \cite{Tzanakis2016} where $q = 4$. Using experimental approaches, we obtained a $\CA(262141; 4, 257, 16)$ by the vertical concatenation of four strength-$3$ covering arrays employing all $16^2 +1$ points of the ovoid. This observation strengthens the idea of finding a similar construction for even prime powers, which is also one of our future work plans.

\section*{Acknowledgements}
The authors would like to thank an anonymous referee for the substantial technical
insight provided, which improved the presentation and helped simplify several proofs.

\section*{Funding}

Lucia Moura and Brett Stevens were supported by the Natural Sciences and Engineering Research Council of Canada grants RGPIN-2023-03890 and RGPIN-2023-04668, respectively.  
\bibliography{Ref}

\end{document}